\documentclass[11pt,reqno]{amsart}

\usepackage[a4paper,margin=1.05in]{geometry}
\usepackage{amsmath,amssymb,amsfonts,mathtools}
\usepackage{microtype}
\usepackage{hyperref}

\hypersetup{
  colorlinks=true,
  linkcolor=blue,
  citecolor=blue,
  urlcolor=blue,
  pdftitle={Projectional Continuous Data Assimilation on the Torus:
    Resonant and Kernel-Free Regimes},
  pdfauthor={Tsogtgerel Gantumur},
  pdfsubject={Continuous data assimilation for the two-dimensional
    Navier--Stokes equations},
  pdfkeywords={continuous data assimilation, Navier--Stokes equations,
    projectional nudging, observability, synchronization}
}

\theoremstyle{plain}
\newtheorem{theorem}{Theorem}[section]
\newtheorem{lemma}[theorem]{Lemma}
\newtheorem{proposition}[theorem]{Proposition}
\newtheorem{corollary}[theorem]{Corollary}

\theoremstyle{definition}
\newtheorem{definition}[theorem]{Definition}

\newtheorem{example}[theorem]{Example}

\theoremstyle{remark}
\newtheorem{remark}[theorem]{Remark}

\numberwithin{equation}{section}

\newcommand{\R}{\mathbb{R}}
\newcommand{\T}{\mathbb{T}}
\newcommand{\Z}{\mathbb{Z}}
\newcommand{\dd}{\,\mathrm{d}}
\newcommand{\esssup}{\operatorname*{ess\,sup}}

\title[Projectional CDA for 2D NSE]
{Projectional Continuous Data Assimilation on the Torus:
Resonant and Kernel-Free Regimes}

\author{Tsogtgerel Gantumur}

\address{McGill University, Montr\'{e}al, QC, Canada}
\address{National University of Mongolia, Ulaanbaatar, Mongolia}
\address{Institute of Mathematics and Digital Technology, Mongolian Academy of Sciences, Ulaanbaatar, Mongolia}
\email{gantumur.tsogtgerel@mcgill.ca}

\date{July 18, 2026}



\begin{document}

\begin{abstract}
We study continuous data assimilation for the two-dimensional incompressible
Navier--Stokes equations on the periodic torus using a single signed scalar
velocity projection in a prescribed spatially varying direction.  We prove
exponential synchronization in both \(L^2\) and \(H^1\) through two
complementary mechanisms.  In the regular resonant regime, nontrivial invisible
currents are controlled through a moving-frame expansion, adapted coordinates,
and Farhat--Lunasin--Titi logarithmic estimates, provided the resulting geometric
shear defect is viscously absorbable.  This recovers the periodic one-component
mechanism for constant rational directions and applies to genuinely nonconstant
projection fields.  In the kernel-free regime, qualitative injectivity and
compactness yield observability with arbitrarily small viscous leakage, giving
synchronization for every sufficiently large gain without an upper gain
restriction.
For sufficiently regular projection fields, both mechanisms extend to
\(L^2\)-stable Type-I coarse scalar observations satisfying a first-order approximation property, under the usual
gain--resolution condition.  A common parabolic smoothing argument then upgrades
all four \(L^2\)-synchronization results to \(H^1\)-synchronization without
additional observation hypotheses.
\end{abstract}

\maketitle
\setcounter{tocdepth}{1}
\tableofcontents

\section{Introduction}
\label{sec:intro}

Continuous data assimilation (CDA), or nudging, reconstructs an unknown reference
trajectory of a dissipative evolution equation by evolving a second copy of the
system with feedback driven by observations.  For the two-dimensional
Navier--Stokes equations, the Azouani--Olson--Titi (AOT) algorithm established
exponential synchronization from sufficiently resolved coarse velocity data
\cite{AOT2014}.  Farhat, Lunasin, and Titi (FLT) subsequently showed that, on a
two-dimensional periodic domain, observations of only one velocity component can
suffice \cite{FLT2016}.  The CDA framework has since been extended to nonlinear
feedback, dynamic and mobile observers, recovery of unknown body forces,
higher-order synchronization, and mesh-free interpolant observables
\cite{CarlsonLariosTiti2024,FranzLariosVictor2022,
BiswasBradshawJolly2023,FarhatLariosMartinezWhitehead2024,
BiswasMartinez2017,BiswasBrownMartinez2022}.

This paper studies scalar projectional observations for the periodic
two-dimensional Navier--Stokes equations.  Instead of observing a fixed Cartesian
component, we prescribe a spatially varying unit direction field \(m\) and observe
the signed scalar projection \(u\cdot m\) of the reference velocity.  Writing
\(e=v-u\) and \(r=e\cdot m\), the exact-data analysis begins from the error identity
\begin{equation}
\label{eq:error-energy-intro}
  \frac12\frac{\dd}{\dd t}|e|^2
  +
  \nu\|e\|^2
  +
  \mu|r|^2
  =
  -b(e,u,e),
\end{equation}
where \(b(a,b,c)=\int_{\mathbb T^2}(a\cdot\nabla b)\cdot c\,dx\)
is the usual Navier--Stokes trilinear form.  The central question is how the
observed scalar dissipation, together with viscosity, controls the nonlinear
production term on the right-hand side.

The answer depends on the solenoidal kernel of the scalar observation.  We develop
two complementary synchronization mechanisms.

In the \emph{resonant regime}, nonzero divergence-free mean-zero error fields are
invisible to the observation.  Constant rational projection directions provide
the basic example: periodic transverse shear modes are unobserved.  The classical
one-component FLT theorem shows that static observability is nevertheless not
necessary, because the Navier--Stokes nonlinearity can transfer information between
the observed and unobserved components.

We call a resonant projection field \emph{regular} when its transverse flow can be
straightened in periodic coordinates on a controlled finite cover of the torus.
For such fields, a moving-frame expansion reduces every non-defect contribution to
weighted variants of the FLT logarithmic forms containing the observed scalar
\(r\).  Spatial variation of the frame also produces a pure transverse term,
measured by a geometric shear defect.  When this defect is small enough to be
absorbed by viscosity, the logarithmic contribution can be balanced against the
scalar feedback in \eqref{eq:error-energy-intro}, yielding exponential
synchronization for every sufficiently large gain.  

The second mechanism applies in the \emph{kernel-free regime}, where no nonzero
solenoidal field is invisible to the scalar projection.  Compactness of the
embedding of divergence-free \(H^1\) fields into \(L^2\) upgrades this qualitative
injectivity to an observability estimate with arbitrarily small viscous leakage.
The way this estimate is used is important: the leakage controls only the
nonlinear production term in \eqref{eq:error-energy-intro}, rather than weakening
the feedback itself.  It is therefore multiplied by a reference-flow bound, not
by the nudging gain.  Exact kernel-free observations consequently synchronize
for every sufficiently large gain, with no upper gain restriction.

For constant directions, the distinction between the two mechanisms is
arithmetic.  Rational transverse directions admit periodic invisible shears and
belong to the resonant theory, whereas nonresonant directions are kernel-free.
Small divisors may enter the compact-observability constant in the latter case,
but the proof does not require an explicit inversion estimate.

We also treat Type-I coarse scalar observations, in which the observed
projection is processed by an \(L^2\)-stable operator satisfying a first-order
approximation property.  This class includes low Fourier projections, local or
mollified volume averages, and suitable \(L^2\)-stable quasi-interpolants.  A
common scalar coercivity estimate reduces the coarse analysis to the
corresponding exact-data mechanism, at the cost of the familiar
gain--resolution condition \(\mu h^2\lesssim\nu\).  In the kernel-free regime,
the approximation argument additionally requires \(W^{1,\infty}\)-regularity
of the prescribed projection field, so that the observed scalar has the
\(H^1\)-regularity needed by the Type-I estimate.

To keep the distinction between the resonant and kernel-free regimes, and their
two analytic mechanisms, in focus, we restrict the coarse theory to Type-I
observations.  Type-II observations, including raw nodal reconstructions, produce
second-order interpolation leakage and require a mixed strong-energy analysis;
we leave that extension to future work.  By contrast, once exponential
synchronization in \(L^2\) has been established, the bounded feedback can be
treated as a lower-order forcing in the periodic strong-error estimate.  The
resulting parabolic argument upgrades all four synchronization theorems to
exponential decay in \(H^1\).

The two mechanisms are complementary, but the present theorems do not exhaust the resonant class. 
A general resonant field may have more complicated transverse dynamics, possibly decomposing into invariant components with invisible currents supported on only some of them. Identifying the appropriate finer classification, and determining whether nonlinear transfer on the resonant components can be combined with compact observability on the complementary ones, would be an interesting problem. Broader kernel-free theories on surfaces and bounded domains, including localized rank-one observations and boundary anchors, are also left for subsequent work.

The paper is organized as follows.  Section~\ref{sec:setting} introduces the
periodic functional setting, the exact and Type-I projectional nudging schemes,
and their error identities.  Section~\ref{sec:geometry} develops invisible
currents, moving frames, and controlled adapted coverings.
Section~\ref{sec:energy-estimates} proves the logarithmic nonlinear compatibility
estimate, and Section~\ref{sec:conv} establishes exact synchronization in the
regular resonant regime.  Section~\ref{sec:kernel-free} develops compact
observability and exact kernel-free synchronization.
Section~\ref{sec:coarse} treats Type-I coarse observations in both regimes, and
Section~\ref{sec:V-decay} upgrades all four \(L^2\)-synchronization results to
\(H^1\)-synchronization.  Section~\ref{sec:concluding-remarks} concludes, while
Appendix~\ref{sec:wellposed} records bounded-feedback well-posedness.

\section{Preliminaries and projectional observations}
\label{sec:setting}

We work on the flat two-dimensional torus
\(\T^2=(\R/2\pi\Z)^2\),
and all functions are \(2\pi\)-periodic in both variables.  We restrict throughout
to mean-zero vector fields, as is natural for the periodic Navier--Stokes equations
after subtracting the conserved spatial mean.

\subsection{Periodic functional setting}

Define the smooth solenoidal test space by
\[
  \mathcal V
  =
  \Bigl\{
    \varphi\in C^\infty(\T^2;\R^2):
    \nabla\cdot\varphi=0,\,
    \int_{\T^2}\varphi\,dx=0
  \Bigr\}.
\]
Here and below, functions on \(\T^2\) may equivalently be viewed as
\(2\pi\)-periodic functions on \(\R^2\).  We then set
\[
  H=\overline{\mathcal V}^{\,L^2(\T^2)^2},
  \qquad
  V=\overline{\mathcal V}^{\,H^1(\T^2)^2}.
\]
Equivalently, one has
\[
  H
  =
  \Bigl\{
    u\in L^2(\T^2)^2:
    \nabla\cdot u=0,\,
    \int_{\T^2}u\,dx=0
  \Bigr\},
  \qquad
  V=H^1(\T^2)^2\cap H.
\]

The \(L^2\) inner product and norm are denoted by \((\cdot,\cdot)\) and
\(|\cdot|\), respectively.  For \(s\ge0\), \(\|\cdot\|_{H^s}\) denotes the
usual Sobolev norm on \(\T^2\).  Following standard Navier--Stokes notation,
we use \(\|u\|:=\|\nabla u\|_{L^2}\) for \(u\in V\).  Although this
quantity is only a seminorm on \(H^1(\T^2)^2\), it is a norm on \(V\),
equivalent to the full \(H^1\)-norm.  Indeed, Poincare's inequality gives
\begin{equation}
\label{eq:poincare}
  \lambda_1|u|^2\le\|u\|^2,
  \qquad
  u\in V.
\end{equation}
With the present normalization of \(\T^2\), one may take \(\lambda_1=1\).

Denote by \(P_\sigma:L^2(\T^2)^2\to H\) the orthogonal projection onto the
mean-zero solenoidal subspace.  The Stokes operator is defined by
\begin{equation}
\label{eq:Stokes-periodic}
  A=-P_\sigma\Delta,
  \qquad
  D(A)=H^2(\T^2)^2\cap H.
\end{equation}
Since the Laplacian preserves both incompressibility and zero mean, one has
\(Au=-\Delta u\) for \(u\in D(A)\).  Consequently,
\(|Au|=\|\Delta u\|_{L^2}\), which is equivalent to the \(H^2\)-norm on
\(D(A)\).

For smooth solenoidal fields, define
\[
  B(u,v)=P_\sigma\bigl((u\cdot\nabla)v\bigr),
  \qquad
  b(u,v,w)
  =
  \int_{\T^2}(u\cdot\nabla v)\cdot w\,dx.
\]
The bilinear map \(B\) extends continuously from \(V\times V\) into \(V'\),
with
\[
  \langle B(u,v),w\rangle_{V',V}=b(u,v,w).
\]
For divergence-free \(u\), periodic integration by parts gives
\begin{equation}
\label{eq:cancellation}
  b(u,v,v)=0,
\end{equation}
and, more generally,
\begin{equation}
\label{eq:skew}
  b(u,v,w)=-b(u,w,v).
\end{equation}

We shall use the standard two-dimensional Sobolev and interpolation inequalities
without further comment.  In particular, Ladyzhenskaya's inequality followed by
Young's inequality gives
\begin{equation}
\label{eq:basic-nonlinear}
  |b(e,u,e)|
  \le
  \frac{\nu}{4}\|e\|^2
  +
  \frac{C}{\nu}\|u\|^2|e|^2,
  \qquad
  e,u\in V.
\end{equation}
The sharper logarithmic compatibility estimate required in the regular resonant
regime is proved in Section~\ref{sec:energy-estimates}.

\subsection{Reference solution}

The reference solution \(u\) satisfies the periodic two-dimensional
Navier--Stokes equations
\begin{equation}
\label{eq:NSE}
  u_t+\nu Au+B(u,u)=f,
  \qquad
  u(0)=u_0\in V,
\end{equation}
where \(f\in H\) is time-independent.  In primitive variables, the system takes
the form
\[
  u_t-\nu\Delta u+(u\cdot\nabla)u+\nabla p=f,
  \qquad
  \nabla\cdot u=0,
  \qquad
  \int_{\T^2}u\,dx=0.
\]
For every \(u_0\in V\), the system has a unique global
strong solution, which eventually enters bounded absorbing sets
\cite{FoiasManleyRosaTemam,Temam}.  The synchronization theorems below state
explicitly the uniform \(L^2\)- and \(H^1\)-bounds on \(u\) required in their
gain and geometric hypotheses.

\subsection{Exact and coarse scalar projectional nudging}
\label{ss:nudging-setup}

Let \(m:\T^2\to\R^2\) be a prescribed measurable unit vector field, with
\(|m(x)|=1\) for almost every \(x\in\T^2\).  Define the associated rank-one
projection tensor by
\begin{equation}
\label{eq:Mdef}
  M(x)=m(x)\otimes m(x).
\end{equation}
Pointwise, one has \(M^2=M\), \(M^\top=M\), and \(0\le M\le I\).
Consequently, \(M\) defines a bounded, self-adjoint, nonnegative multiplication
operator on \(L^2(\T^2)^2\), with operator norm one.  No differentiability of
\(m\) is needed for this operator-theoretic formulation.  The regular resonant
theory will assume \(m\in W^{2,\infty}\), while the exact kernel-free theory
applies to measurable unit fields.

For a nudging gain \(\mu>0\), the exact projectionally nudged system is given by
\begin{equation}
\label{eq:nudged}
  v_t+\nu Av+B(v,v)
  =
  f-\mu P_\sigma M(v-u),
  \qquad
  v(0)=v_0\in V.
\end{equation}
Since \(M(v-u)=((v-u)\cdot m)m\), the feedback uses only the signed scalar
discrepancy \((v-u)\cdot m\).  For \(m=e_1\) or \(m=e_2\), this reduces to the
exact-observation form of one-component nudging.

We also consider a coarse scheme in which the scalar projection is processed by
a bounded linear operator \(I_h\in\mathcal L(L^2(\T^2))\).  For
\(z\in L^2(\T^2)^2\), define
\begin{equation}
\label{eq:Mhdef}
  M_hz
  :=
  \bigl(I_h(z\cdot m)\bigr)m.
\end{equation}
Thus \(M_h\) extracts the scalar component \(z\cdot m\), applies the observation
operator \(I_h\), and reinserts the resulting scalar field in the direction \(m\).
Since \(|m|=1\) almost everywhere, one has
\[
  |M_hz|
  =
  |I_h(z\cdot m)|
  \le
  \|I_h\|_{\mathcal L(L^2)}|z| ,
\]
so \(P_\sigma M_h\) is bounded from \(H\) to \(H\).  In general, \(M_h\)
need not be self-adjoint or nonnegative; neither property is required for the
well-posedness or error-identity arguments below.

The coarse scalar projectionally nudged system is given by
\begin{equation}
\label{eq:coarse-nudged}
  v_t+\nu Av+B(v,v)
  =
  f-\mu P_\sigma M_h(v-u),
  \qquad
  v(0)=v_0\in V.
\end{equation}
Taking \(I_h\) to be the identity recovers the exact system
\eqref{eq:nudged}.  Boundedness of \(I_h\) is sufficient for well-posedness.
In Section~\ref{sec:coarse}, we impose the additional stability and
first-order approximation assumptions defining the Type-I observation class.

\subsection{Well-posedness with bounded feedback}

The existence theory requires only boundedness of the feedback operator.

\begin{proposition}[Bounded-feedback well-posedness]
\label{prop:bounded-feedback-wellposed}
For every \(v_0\in V\), \(\mu\ge0\), and \(T>0\), each of
\eqref{eq:nudged} and \eqref{eq:coarse-nudged} has a unique strong solution
satisfying
\[
  v\in C([0,T];V)\cap L^2(0,T;D(A)),
  \qquad
  v_t\in L^2(0,T;H).
\]
The same conclusion holds for the more general system
\[
  v_t+\nu Av+B(v,v)
  =
  f-\mu P_\sigma\mathcal M(t)(v-u),
\]
provided that
\(\mathcal M\in
L^\infty(0,T;\mathcal L(L^2(\T^2)^2))\)
is strongly measurable.  In particular, the exact and coarse nudged systems are
globally well posed.
\end{proposition}

The proof is the standard Fourier--Galerkin argument for the periodic
two-dimensional Navier--Stokes equations with bounded lower-order feedback; it is
included in Appendix~\ref{sec:wellposed}.  The general formulation also covers
measurable time-dependent projection fields and uniformly bounded
time-dependent scalar observation operators, although the synchronization results
below assume that \(m\) is time-independent.

\subsection{Exact and coarse error identities}
\label{ss:error-identities}

For the remainder of the paper, write \(e:=v-u\) and \(r:=e\cdot m\).
Subtracting \eqref{eq:NSE} from the exact nudged system
\eqref{eq:nudged} gives
\begin{equation}
\label{eq:error}
  e_t+\nu Ae+B(e,u)+B(v,e)
  =
  -\mu P_\sigma Me.
\end{equation}
This form is convenient at the \(L^2\)-level because \(B(v,e)\) is skew
against \(e\).

Testing \eqref{eq:error} with \(e\), using
\eqref{eq:cancellation} and the orthogonality of \(P_\sigma\), gives the
exact error identity
\begin{equation}
\label{eq:L2-error-identity}
  \frac12\frac{d}{dt}|e|^2
  +
  \nu\|e\|^2
  +
  \mu|r|^2
  =
  -b(e,u,e).
\end{equation}
Here we used
\((Me,e)=\int_{\T^2}|e\cdot m|^2\,dx=|r|^2\).
Identity \eqref{eq:L2-error-identity} is the common starting point for
the exact regular resonant theorem in Section~\ref{sec:conv} and the exact
kernel-free theorem in Section~\ref{sec:kernel-free}.

For the coarse system \eqref{eq:coarse-nudged}, subtraction of
\eqref{eq:NSE} yields
\begin{equation}
\label{eq:coarse-error}
  e_t+\nu Ae+B(e,u)+B(v,e)
  =
  -\mu P_\sigma M_he.
\end{equation}
Since \(M_he=(I_hr)m\), testing \eqref{eq:coarse-error} with \(e\) gives
the coarse error identity
\begin{equation}
\label{eq:coarse-L2-error-identity}
  \frac12\frac{d}{dt}|e|^2
  +
  \nu\|e\|^2
  +
  \mu(I_hr,r)
  =
  -b(e,u,e).
\end{equation}
No self-adjointness of \(I_h\) is needed, because
\((P_\sigma M_he,e)=(M_he,e)=(I_hr,r)\).

Thus the exact and coarse analyses have the same nonlinear production term.
Only the scalar dissipation changes, from \(|r|^2\) to \((I_hr,r)\).
The common Type-I coercivity estimate for the latter is proved in
Section~\ref{sec:coarse}.

\section{Moving-frame geometry and projectional regimes}
\label{sec:geometry}

This section distinguishes two geometric features of a scalar projection field:
the possible existence of nonzero divergence-free currents invisible to the
observation, and the transverse geometry that supports the logarithmic
Farhat--Lunasin--Titi mechanism when such currents are present.  The kernel-free
mechanism is developed in Section~\ref{sec:kernel-free}; here we introduce the
moving-frame and adapted-covering geometry used in the regular resonant regime.

We retain the prescribed unit field \(m\) from \S\ref{ss:nudging-setup} and set
\(n:=m^\perp=(-m_2,m_1)\).  For a scalar function \(F\), write
\(D_mF:=m\cdot\nabla F\) and \(D_nF:=n\cdot\nabla F\).  For vector fields, these
directional derivatives are applied componentwise.  The regularity assumptions on
\(m\) are stated where they are needed; the regular resonant theory ultimately
uses \(m\in W^{2,\infty}(\T^2;\R^2)\).

\subsection{Invisible currents and the two mechanisms}
\label{subsec:invisible-currents}

The solenoidal kernel of the exact scalar observation \(e\mapsto e\cdot m\) is
defined by
\begin{equation}
\label{eq:invisible-current-space}
  \mathcal K_m
  :=
  \bigl\{
    e\in H:\ e\cdot m=0
  \bigr\}.
\end{equation}
We call its elements \emph{invisible currents}.  Since \(m,n\) form an
orthonormal frame, every \(e\in\mathcal K_m\) can be written as
\(e=\sigma n\), where \(\sigma=e\cdot n\).  If
\(m\in W^{1,\infty}\) and \(e\in\mathcal K_m\cap V\), incompressibility gives,
in the distributional sense,
\begin{equation}
\label{eq:invisible-current-transport}
  D_n\sigma+(\nabla\cdot n)\sigma=0.
\end{equation}
Thus invisible currents are transported along the transverse \(n\)-flow, with
a weight determined by its divergence.

This geometry yields a dichotomy between kernel-free and
kernel-nontrivial observations, together with two corresponding synchronization
mechanisms.

\begin{itemize}
\item If \(\mathcal K_m=\{0\}\), the observation is
\emph{kernel-free}.  Compactness of \(V\hookrightarrow H\) then yields
observability with arbitrarily small viscous leakage and, consequently, the
large-gain synchronization theorem of Section~\ref{sec:kernel-free}.

\item If \(\mathcal K_m\neq\{0\}\), invisible currents are present.  In this
\emph{resonant} case, synchronization may still follow from nonlinear coupling.
The argument in
\S\ref{sec:energy-estimates}--\S\ref{sec:conv} applies to the regular resonant
subclass, where the transverse flow admits a controlled adapted covering and
the associated shear defect is viscously absorbable.  More general resonant
fields remain outside the present analysis.
\end{itemize}

\subsection{Frame derivatives and shear defect}
\label{subsec:frame-defect}

The derivatives of the oriented orthonormal frame \(m,n\) are encoded by the
connection coefficients
\(a_m:=(D_m m)\cdot n\) and \(a_n:=(D_n m)\cdot n\).
Differentiating \(|m|^2=|n|^2=1\) and \(m\cdot n=0\) gives, almost
everywhere,
\begin{equation}
\label{eq:frame-directional-connection}
  D_m m=a_m n,
  \qquad
  D_m n=-a_m m,
  \qquad
  D_n m=a_n n,
  \qquad
  D_n n=-a_n m.
\end{equation}
These coefficients enter as lower-order terms in the moving-frame expansion of
the Navier--Stokes nonlinearity.  The contribution containing no factor of the
observed scalar \(r=e\cdot m\) is governed by the following defect.

\begin{definition}[Shear-defect vector]
\label{def:shear-defect}
For \(m\in W^{1,\infty}(\T^2;\R^2)\), define the associated
\emph{shear-defect vector} by
\begin{equation}
\label{eq:shear-defect-vector}
  \mathcal A_m
  :=
  D_n n-(\nabla\cdot n)n.
\end{equation}
\end{definition}

By \eqref{eq:frame-directional-connection}, one has
\(\mathcal A_m=-a_n m-(\nabla\cdot n)n\).  Consequently, if
\(u=Um+Vn\), where \(U=u\cdot m\) and \(V=u\cdot n\), then
\begin{equation}
\label{eq:u-dot-shear-defect}
  -u\cdot\mathcal A_m
  =
  a_nU+(\nabla\cdot n)V.
\end{equation}
This identity describes the coefficient of the pure transverse-shear term in the
nonlinear expansion below.

Since \(n=m^\perp\), one has
\(\|\nabla n\|_{L^\infty}=\|\nabla m\|_{L^\infty}\), and therefore
\[
  \|\mathcal A_m\|_{L^\infty}
  \le
  C\|\nabla m\|_{L^\infty}
\]
for a universal constant \(C\).  In particular, if \(m\) is constant, then
\(a_m=a_n=0\) and \(\mathcal A_m\equiv0\).

\subsection{Controlled adapted coverings}
\label{subsec:adapted-coverings}

We now introduce coordinates that straighten the transverse direction
\(n=m^\perp\).  Let
\(\widehat{\T}^2=(\R/L_y\Z)\times(\R/L_s\Z)\), with \(L_y,L_s>0\)
and periodic coordinates \(Y=(y,s)\), and let
\(\Phi:\widehat{\T}^2\to\T^2\) be a smooth finite covering map.  Since
\(\Phi\) is a local diffeomorphism, we may write
\(\Gamma(Y):=(D\Phi(Y))^{-1}\).

Define the coordinate components \(\beta_y,\beta_s\) of \(m\) by
\(\Gamma(m\circ\Phi)=\beta_y\partial_y+\beta_s\partial_s\).
Then every scalar function \(F\) satisfies
\begin{equation}
\label{eq:Dm-pullback-beta}
  (D_mF)\circ\Phi
  =
  \beta_y\,\partial_y(F\circ\Phi)
  +
  \beta_s\,\partial_s(F\circ\Phi).
\end{equation}

\begin{definition}[Controlled adapted covering]
\label{def:controlled-adapted-covering}
A unit field \(m\in W^{2,\infty}(\T^2;\R^2)\) is said to admit a
\emph{controlled adapted covering} if there exist
\(\widehat{\T}^2\) and \(\Phi\) as above, together with a positive function
\(q\in W^{2,\infty}(\widehat{\T}^2)\), such that
\begin{equation}
\label{eq:adapted-covering-vector}
  \partial_s\Phi=q\,n\circ\Phi.
\end{equation}
We also require the quantitative bound
\[
  \mathfrak C_\Phi
  :=
  \|D\Phi\|_{W^{2,\infty}}
  +
  \|\Gamma\|_{W^{2,\infty}}
  +
  \|q\|_{W^{2,\infty}}
  +
  \|q^{-1}\|_{W^{2,\infty}}
  <\infty.
\]
Constants depending only on the covering degree and
\(\mathfrak C_\Phi\) are denoted by \(C_\Phi\), while any additional
dependence on bounds for \(m\) will be indicated explicitly.
\end{definition}

For a fixed controlled adapted covering, set
\(J:=|\det D\Phi|\) and \(\rho:={J}/{q}\).

\begin{lemma}[Controlled adapted coverings imply resonance]
\label{lem:adapted-covering-implies-resonance}
If \(m\) admits a controlled adapted covering, then we have
\(\mathcal K_m\neq\{0\}\).
\end{lemma}

\begin{proof}
For \(\chi\in C^\infty(\R/L_y\Z)\), define a scalar function on the base
torus by
\[
  \sigma_\chi(x)
  :=
  \sum_{Y=(y,s)\in\Phi^{-1}(x)}
  \frac{\chi(y)q(Y)}{J(Y)},
\]
and set \(e_\chi:=\sigma_\chi n\).  The sum over the finite fiber makes
\(\sigma_\chi\) well defined on the base.  For every smooth periodic scalar
function \(\varphi\), the area formula and
\eqref{eq:adapted-covering-vector} give
\[
  \int_{\T^2}e_\chi\cdot\nabla\varphi\,dx
  =
  \int_{\widehat{\T}^2}
    \chi(y)q\,(n\cdot\nabla\varphi)\circ\Phi\,dY
  =
  \int_{\widehat{\T}^2}
    \chi(y)\partial_s(\varphi\circ\Phi)\,dY
  =
  0.
\]
Thus \(e_\chi\) is divergence-free and \(e_\chi\cdot m=0\).

For each \(y\), the loop \(s\mapsto\Phi(y,s)\) has a homology vector
\(k\in\Z^2\) independent of \(y\).  Hence
\[
  \int_0^{L_s}q(y,s)n(\Phi(y,s))\,ds
  =
  2\pi k,
\]
where the integral is computed using any lift of the loop to \(\R^2\).
Another application of the area formula therefore yields
\[
  \int_{\T^2}e_\chi\,dx
  =
  2\pi k\int_0^{L_y}\chi(y)\,dy.
\]
It follows that \(e_\chi\) has zero mean whenever \(\chi\) does.

It remains to choose such a \(\chi\) with \(e_\chi\not\equiv0\).  If
\(e_\chi=0\) for every smooth zero-mean \(\chi\), then, for each
\(x\in\T^2\), the finite positive atomic measure
\[
  \sum_{Y=(y,s)\in\Phi^{-1}(x)}
  \frac{q(Y)}{J(Y)}\,\delta_y
\]
would annihilate every smooth zero-mean function on \(\R/L_y\Z\).  It
would therefore be a constant multiple of Lebesgue measure, which is
impossible for a nonzero finite atomic measure.  Thus some smooth
zero-mean \(\chi\) gives \(e_\chi\neq0\).  For this choice,
\(e_\chi\in\mathcal K_m\setminus\{0\}\).
\end{proof}

In view of Lemma~\ref{lem:adapted-covering-implies-resonance}, we call a
projection field \emph{regular resonant} when it admits a controlled adapted
covering.

The adaptedness identity means that the \(s\)-coordinate curves are integral
curves of \(n\), up to the positive reparametrization \(q\).  Thus the
transverse flow becomes a periodic coordinate direction on a finite cover.
This is a global structural hypothesis, not a smallness condition; the later
smallness assumption concerns the shear defect \(\mathcal A_m\).  The covering
need not be unique, and all constants below are relative to one chosen
controlled covering.  We neither select a canonical covering nor optimize over
all possible choices.

The definition and compactness of the cover imply that \(q,J,\rho\) are
uniformly bounded above and away from zero.  Moreover, their
\(W^{2,\infty}\)-norms, as well as those of their reciprocals, are controlled
by \(C_\Phi\).  The same bounds give
\[
  \|\beta_y\|_{W^{2,\infty}}
  +
  \|\beta_s\|_{W^{2,\infty}}
  \le
  C_\Phi\|m\|_{W^{2,\infty}}.
\]

The adapted coordinates also describe the full family of invisible currents.
If \(e=\sigma n\) is an invisible current and
\(\widehat\sigma:=\sigma\circ\Phi\), then the coordinate formula for
divergence gives, in the distributional sense,
\[
  \bigl(\nabla\cdot(\sigma n)\bigr)\circ\Phi
  =
  \frac{1}{J}\,
  \partial_s\bigl(\rho\widehat\sigma\bigr).
\]
Consequently, every invisible divergence-free current pulled back from the
base has the weighted-shear form
\[
  \rho(y,s)\widehat\sigma(y,s)=\chi(y)
\]
for some periodic function \(\chi\).  Conversely, any such weighted shear
that is compatible with the identifications induced by the covering descends
to an invisible divergence-free field on the base; membership in \(H\) also
requires zero spatial mean.  This is the variable-frame analogue of the
periodic shear modes associated with constant rational directions.

The following elementary pullback identities will be used in the nonlinear
estimates.  For a scalar function \(F\) on \(\T^2\), write
\(\widehat F:=F\circ\Phi\).

\begin{lemma}[Directional derivatives in adapted coordinates]
\label{lem:directional-derivatives-adapted-cover}
Assume that \(m\) admits a controlled adapted covering.  Then every
\(F\in H^1(\T^2)\) satisfies, almost everywhere,
\[
  (D_nF)\circ\Phi
  =
  q^{-1}\partial_s\widehat F,
  \qquad
  (D_mF)\circ\Phi
  =
  \beta_y\,\partial_y\widehat F
  +
  \beta_s\,\partial_s\widehat F.
\]
\end{lemma}

\begin{proof}
The adaptedness condition \eqref{eq:adapted-covering-vector} gives
\(n\circ\Phi=q^{-1}\partial_s\Phi\).  The first identity therefore follows
from the chain rule, while the second is
\eqref{eq:Dm-pullback-beta}.
\end{proof}

A controlled finite covering preserves Sobolev norms up to constants determined
by its geometry.

\begin{lemma}[Norm comparison under a controlled adapted covering]
\label{lem:cover-norm-comparison}
Let \(\Phi:\widehat{\T}^2\to\T^2\) be the covering map of a controlled
adapted covering.  For \(0\le k\le2\) and every scalar or vector field
\(F\in H^k(\T^2)\), one has
\[
  C_\Phi^{-1}\|F\|_{H^k(\T^2)}
  \le
  \|F\circ\Phi\|_{H^k(\widehat{\T}^2)}
  \le
  C_\Phi\|F\|_{H^k(\T^2)}.
\]
\end{lemma}

\begin{proof}
For \(k=0\), the area formula gives
\[
  \int_{\widehat{\T}^2}|F\circ\Phi|^2\,dY
  =
  \int_{\T^2}|F(x)|^2
  \sum_{Y\in\Phi^{-1}(x)}\frac{1}{J(Y)}\,dx.
\]
The two-sided estimate follows from the finite covering degree and the positive
upper and lower bounds for \(J\).  For \(k=1,2\), the upper bound follows from
the chain rule and the \(W^{2,\infty}\)-bounds for \(D\Phi\).  Applying the
same argument on a finite collection of inverse branches, whose derivatives are
controlled by \(\Gamma=(D\Phi)^{-1}\) and its derivatives, gives the lower
bound.
\end{proof}

\subsection{Examples of adapted coverings}
\label{subsec:adapted-covering-examples}

We first record the constant rational case, then describe its stability under
controlled changes of coordinates and an explicit nonconstant resonant family.

\begin{example}[Constant rational directions]
\label{ex:constant-rational-directions}
Let \(a\in\mathbb S^1\) be constant, set \(b=a^\perp\), and suppose that
\(b=\pm\ell/|\ell|\) for a primitive
\(\ell\in\mathbb Z^2\setminus\{0\}\).  With \(L:=|\ell|\), consider
\(\widehat{\T}^2=(\R/2\pi L\Z)^2\) and define
\[
  \Phi(y,s)=ya+sb
  \quad\mod 2\pi\mathbb Z^2.
\]
This is a finite covering of degree \(L^2\), and
\(\partial_s\Phi=b=n\circ\Phi\).  Thus \(q\equiv1\), so the
controlled adapted-covering hypothesis holds.  For the coordinate directions,
one may take \(L=1\) and \(\Phi\) to be the identity.

These fields are resonant.  Indeed, if \(\psi\) is a nonzero mean-zero
\(2\pi\)-periodic function, then
\(e(x):=\psi(\ell^\perp\cdot x)b\) is periodic, divergence-free, mean-zero,
and orthogonal to \(m=a\).  Hence
\(\mathcal K_m\neq\{0\}\).
\end{example}

The construction is stable under controlled changes of coordinates.  Let
\(\Phi_0\) be the preceding rational cover and let
\(\Theta:\T^2\to\T^2\) be a \(C^3\)-diffeomorphism.  Define
\[
  \Phi:=\Theta\circ\Phi_0,
  \qquad
  n(\Theta(x))
  :=
  \frac{D\Theta(x)b}{|D\Theta(x)b|},
  \qquad
  m:=-n^\perp.
\]
Then
\(\partial_s\Phi=q\,n\circ\Phi\), where
\(q(y,s):=|D\Theta(\Phi_0(y,s))b|>0\).  Hence \(m\) admits a
controlled adapted covering, with constants determined by the covering degree
and suitable \(C^3\)-bounds for \(\Theta\) and \(\Theta^{-1}\).
This gives a stable class of adapted perturbations, but not every small
perturbation of a rational direction is covered: without an adapted fibration,
the transverse orbits need not close on a finite cover.

\begin{example}[An explicit sinusoidal perturbation]
\label{ex:explicit-sinusoidal-perturbation}
For \(|\varepsilon|\le\frac12\), consider the volume-preserving
diffeomorphism
\[
  \Phi_\varepsilon(y,s)
  =
  \bigl(y+\varepsilon\sin s,s\bigr).
\]
Set
\[
  q_\varepsilon(x_2)
  :=
  \sqrt{1+\varepsilon^2\cos^2x_2},
  \qquad
  n_\varepsilon(x)
  :=
  \frac{(\varepsilon\cos x_2,1)}{q_\varepsilon(x_2)},
  \qquad
  m_\varepsilon(x)
  :=
  \frac{(1,-\varepsilon\cos x_2)}{q_\varepsilon(x_2)}.
\]
Then \(n_\varepsilon=m_\varepsilon^\perp\),
\(\det D\Phi_\varepsilon=1\), and
\(\partial_s\Phi_\varepsilon
=q_\varepsilon n_\varepsilon\circ\Phi_\varepsilon\).
Thus \(\Phi_\varepsilon\) is an adapted covering of degree one.  Direct
differentiation shows that its controlled-covering constants and
\(\|m_\varepsilon\|_{W^{2,\infty}}\) are bounded uniformly for
\(|\varepsilon|\le\frac12\).

The associated shear defect is given explicitly by
\[
  \mathcal A_{m_\varepsilon}(x)
  =
  -\frac{\varepsilon\sin x_2}
  {1+\varepsilon^2\cos^2x_2}\,(1,0),
  \qquad
  \|\mathcal A_{m_\varepsilon}\|_{L^\infty}
  \le
  |\varepsilon|.
\]
The family is genuinely resonant.  For any nonzero mean-zero
\(2\pi\)-periodic function \(\psi\), define
\[
  \xi:=x_1-\varepsilon\sin x_2,
  \qquad
  e_\varepsilon(x)
  :=
  \psi(\xi)(\varepsilon\cos x_2,1).
\]
A direct calculation gives
\(\nabla\cdot e_\varepsilon=0\) and
\(e_\varepsilon\cdot m_\varepsilon=0\), while integration first in
\(x_1\) shows that \(e_\varepsilon\) has zero mean.  Therefore
\(e_\varepsilon\in\mathcal K_{m_\varepsilon}\setminus\{0\}\).

Consequently, \(m_\varepsilon\) lies in the regular resonant regime with
uniformly controlled adapted geometry and a shear defect of order
\(O(|\varepsilon|)\).  For fixed viscosity and reference-flow bounds, the
shear-defect hypothesis of Theorem~\ref{thm:periodic-conv} therefore holds
whenever \(|\varepsilon|\) is sufficiently small.
\end{example}

For comparison, let \(m=a\) be constant and suppose that
\(n=a^\perp\) is not parallel to any nonzero lattice vector.  The \(n\)-orbits
are then dense, so no adapted finite covering of the preceding type exists.
The observation is nevertheless kernel-free, as shown in
Section~\ref{sec:kernel-free}.  The associated compact-observability constants
may encode the small divisors \(k\cdot n\), but the kernel-free argument avoids
their explicit inversion.

\section{Logarithmic nonlinear compatibility in the regular resonant regime}
\label{sec:energy-estimates}

This section establishes the nonlinear estimate underlying the regular resonant mechanism. We assume that \(m\) admits a controlled adapted covering and hence, by Lemma~\ref{lem:adapted-covering-implies-resonance}, is resonant. No smallness of the shear defect is imposed at this stage.

After expanding \(b(e,u,e)\) in the moving frame, every non-defect term contains
the observed scalar \(r=e\cdot m\).  These terms are pulled back to the adapted
cover and estimated by weighted logarithmic trilinear inequalities.  The sole
term without a factor of \(r\) is governed by \(\mathcal A_m\).  Its viscous
absorbability is imposed later, in Section~\ref{sec:conv}.

\subsection{Logarithmic trilinear estimates}

For \(H\in H^1(\widehat{\T}^2)\), define the logarithmic factor by
\begin{equation}
\label{eq:log-factor}
  \mathcal L(H)
  :=
  \begin{cases}
  \displaystyle
  1+\log\left(
    1+
    \frac{\|H\|_{H^1(\widehat{\T}^2)}^2}
         {|H|_{L^2(\widehat{\T}^2)}^2}
  \right),
  & H\not\equiv0,\\[2ex]
  1,
  & H\equiv0.
  \end{cases}
\end{equation}
We first extend the logarithmic trilinear estimate used by
Farhat--Lunasin--Titi from mean-zero functions to arbitrary \(H^1\)-functions
on a flat torus.

\begin{lemma}[Logarithmic trilinear estimate]
\label{lem:Titi-FLT-core}
For every flat two-dimensional torus \(\widehat{\T}^2\), there exists
\(C_{\log}>0\) such that, for all
\(F,G,H\in H^1(\widehat{\T}^2)\), one has
\begin{equation}
\label{eq:Titi-FLT-core}
  \left|
    \int_{\widehat{\T}^2}
      F\,\partial_jG\,H\,dY
  \right|
  \le
  C_{\log}
  \|F\|_{H^1}
  \|G\|_{H^1}
  |H|_{L^2}\,
  \mathcal L(H)^{1/2},
  \qquad
  j=1,2.
\end{equation}
\end{lemma}

\begin{proof}
After a fixed linear identification, it suffices to work on a square torus;
the resulting constant depends only on the flat geometry.  Write
\(F=F_0+\bar F\), \(G=G_0+\bar G\), and
\(H=H_0+\bar H\), where the subscript \(0\) denotes the zero-mean
part.  Since \(\partial_jG=\partial_jG_0\) and
\(\int_{\widehat{\T}^2}\partial_jG_0\,dY=0\), the integral in
\eqref{eq:Titi-FLT-core} equals
\[
  \int_{\widehat{\T}^2}F_0\,\partial_jG_0\,H_0\,dY
  +
  \bar F
  \int_{\widehat{\T}^2}\partial_jG_0\,H_0\,dY
  +
  \bar H
  \int_{\widehat{\T}^2}F_0\,\partial_jG_0\,dY.
\]

The mean-zero FLT estimate \cite{FLT2016} bounds the first term by
\[
  C\|F_0\|_{H^1}\|G_0\|_{H^1}|H_0|_{L^2}
  \left[
    1+\log\left(
      1+C\frac{\|H_0\|_{H^1}^2}{|H_0|_{L^2}^2}
    \right)
  \right]^{1/2},
\]
with the expression interpreted as zero when \(H_0\equiv0\).
The zero-mean projection is bounded on \(L^2\) and \(H^1\), and the function
\(x\mapsto x^2[1+\log(1+A^2/x^2)]\) is increasing for \(x>0\).
Consequently, this term is bounded by the right-hand side of
\eqref{eq:Titi-FLT-core}.  Cauchy--Schwarz gives the same bound for the
two mean terms, since \(\mathcal L(H)\ge1\).
\end{proof}

The moving-frame expansion also produces bounded coefficients and terms without
a derivative.  The form needed below is recorded next.

\begin{lemma}[Weighted logarithmic trilinear estimates]
\label{lem:weighted-log-trilinear}
Let \(A\in W^{1,\infty}(\widehat{\T}^2)\).  There exists a constant
\(C_A\), depending only on the flat geometry and
\(\|A\|_{W^{1,\infty}}\), such that, for all
\(F,G,H\in H^1(\widehat{\T}^2)\), one has
\[
  \left|
    \int_{\widehat{\T}^2}
      A\,F\,\partial_jG\,H\,dY
  \right|
  \le
  C_A
  \|F\|_{H^1}
  \|G\|_{H^1}
  |H|_{L^2}\,
  \mathcal L(H)^{1/2},
  \qquad
  j=1,2,
\]
and
\[
  \left|
    \int_{\widehat{\T}^2}AFGH\,dY
  \right|
  \le
  C_A
  \|F\|_{H^1}
  \|G\|_{H^1}
  |H|_{L^2}\,
  \mathcal L(H)^{1/2}.
\]
\end{lemma}

\begin{proof}
The first estimate follows from Lemma~\ref{lem:Titi-FLT-core}, applied with
\(AF\) in place of \(F\), and the multiplier bound
\(\|AF\|_{H^1}\le
C\|A\|_{W^{1,\infty}}\|F\|_{H^1}\).
For the second, H\"older's inequality and
\(H^1(\widehat{\T}^2)\hookrightarrow L^4(\widehat{\T}^2)\) give
\[
  \left|
    \int_{\widehat{\T}^2}AFGH\,dY
  \right|
  \le
  C\|A\|_{L^\infty}
  \|F\|_{H^1}
  \|G\|_{H^1}
  |H|_{L^2}.
\]
The result follows because \(\mathcal L(H)\ge1\).
\end{proof}

\subsection{Expansion of the trilinear term in the moving frame}

We now isolate the part of the Navier--Stokes nonlinearity that is invisible to
the scalar observation.  The moving-frame expansion shows that the only term
without a factor of \(r=e\cdot m\) is governed by the shear defect
\(\mathcal A_m\).

\begin{lemma}[Moving-frame expansion of the trilinear term]
\label{lem:moving-frame-trilinear-expansion}
Let \(m\in W^{1,\infty}(\T^2;\R^2)\) be a unit vector field, and let
\(e,u\) be smooth periodic divergence-free vector fields.  Retain
\(r=e\cdot m\), and set
\[
  w:=e\cdot n,
  \qquad
  U:=u\cdot m,
  \qquad
  V:=u\cdot n.
\]
Then \(e=rm+wn\), \(u=Um+Vn\), and
\begin{align}
  b(e,u,e)
  &=
  -\int_{\T^2}u\cdot\mathcal A_m\,w^2\,dx
  \notag\\
  &\quad
  +
  \int_{\T^2}
    r^2\bigl(D_mU-a_mV\bigr)\,dx
  \notag\\
  &\quad
  +
  \int_{\T^2}
    rw\bigl(D_nU-D_mV+a_mU-a_nV\bigr)\,dx
  \notag\\
  &\quad
  -
  2\int_{\T^2}rV\,D_mw\,dx.
  \label{eq:moving-frame-trilinear-expansion}
\end{align}
The identity extends to \(e,u\in V\) by density.
\end{lemma}

\begin{proof}
The frame identities \eqref{eq:frame-directional-connection} give
\[
  D_m u
  =
  (D_mU-a_mV)m+(D_mV+a_mU)n,
  \qquad
  D_n u
  =
  (D_nU-a_nV)m+(D_nV+a_nU)n.
\]
Hence the integrand initially expands as
\[
  (e\cdot\nabla u)\cdot e
  =
  r^2(D_mU-a_mV)
  +
  rw(D_nU-a_nV+D_mV+a_mU)
  +
  w^2(D_nV+a_nU).
\]
Only the last term contains no factor of \(r\).

To rewrite it, incompressibility of \(e=rm+wn\) gives
\[
  D_nw
  =
  -D_mr-(\nabla\cdot m)r-(\nabla\cdot n)w.
\]
Periodic integration by parts then yields
\begin{align*}
  \int_{\T^2}w^2D_nV\,dx
  &=
  -2\int_{\T^2}VwD_nw\,dx
  -
  \int_{\T^2}Vw^2\nabla\cdot n\,dx
  \\
  &=
  \int_{\T^2}Vw^2\nabla\cdot n\,dx
  +
  2\int_{\T^2}VwD_mr\,dx
  +
  2\int_{\T^2}Vwr\nabla\cdot m\,dx
  \\
  &=
  \int_{\T^2}Vw^2\nabla\cdot n\,dx
  -
  2\int_{\T^2}rwD_mV\,dx
  -
  2\int_{\T^2}rV D_mw\,dx.
\end{align*}
In the last step, the two terms containing
\(rwV\nabla\cdot m\) cancel.  Since
\(\mathcal A_m=-a_nm-(\nabla\cdot n)n\), it follows that
\[
  \int_{\T^2}w^2(D_nV+a_nU)\,dx
  =
  -\int_{\T^2}u\cdot\mathcal A_m\,w^2\,dx
  -
  2\int_{\T^2}rwD_mV\,dx
  -
  2\int_{\T^2}rV D_mw\,dx.
\]
Substitution into the initial expansion proves
\eqref{eq:moving-frame-trilinear-expansion}.

For \(e,u\in V\), the result follows by approximation with smooth solenoidal
fields, since every term in the identity is continuous under \(H^1\)-convergence
by the two-dimensional embedding \(H^1(\T^2)\hookrightarrow L^4(\T^2)\).
\end{proof}

\subsection{Adapted-coordinate form of the non-defect terms}
\label{subsec:adapted-coordinate-nondefect}

Assume that \(m\) admits a controlled adapted covering, and retain the pullback
notation of \S\ref{subsec:adapted-coverings}.  The key point is that, after
passing to adapted coordinates, every non-defect term in
\eqref{eq:moving-frame-trilinear-expansion} retains an undifferentiated factor
of the observed scalar \(r=e\cdot m\).

\begin{lemma}[Adapted form of the non-defect terms]
\label{lem:adapted-admissible-terms}
Set
\[
  R:=r\circ\Phi,
  \qquad
  \widehat w:=w\circ\Phi,
  \qquad
  \widehat U:=U\circ\Phi,
  \qquad
  \widehat V:=V\circ\Phi.
\]
Then the three non-defect terms in
\eqref{eq:moving-frame-trilinear-expansion} pull back to finite sums of the
forms
\begin{equation}
\label{eq:adapted-derived-term}
  \int_{\widehat{\mathbb T}^2}
    A\,F\,\partial_jG\,R\,dY,
  \qquad
  j=1,2,
\end{equation}
and
\begin{equation}
\label{eq:adapted-lower-term}
  \int_{\widehat{\mathbb T}^2}
    AFG R\,dY.
\end{equation}
Here \(F\) and \(G\) are chosen from
\(\{R,\widehat w,\widehat U,\widehat V\}\).  Each coefficient \(A\) satisfies the \(W^{1,\infty}\)-bound required in
Lemma~\ref{lem:weighted-log-trilinear}.  These bounds depend only on the
covering degree, \(\mathfrak C_\Phi\), and
\(\|m\|_{W^{2,\infty}}\).
\end{lemma}

\begin{proof}
Let \(d_\Phi\) denote the degree of the covering.  The area formula gives
\[
  \int_{\mathbb T^2}f(x)\,dx
  =
  \frac{1}{d_\Phi}
  \int_{\widehat{\mathbb T}^2}
    f(\Phi(Y))J(Y)\,dY.
\]
By Lemma~\ref{lem:directional-derivatives-adapted-cover}, the pullbacks of
\(D_n\) and \(D_m\) are, respectively,
\(q^{-1}\partial_s\) and
\(\beta_y\partial_y+\beta_s\partial_s\).

For example, one of the mixed terms becomes
\[
  \int_{\mathbb T^2}rwD_nU\,dx
  =
  \frac{1}{d_\Phi}
  \int_{\widehat{\mathbb T}^2}
    \frac{J}{q}\,
    \widehat w\,\partial_s\widehat U\,R\,dY,
\]
which has the form \eqref{eq:adapted-derived-term}.  Terms involving \(D_m\)
split similarly into their \(\partial_y\)- and \(\partial_s\)-parts.  In every
case, one factor \(R\) remains undifferentiated.  The terms containing
\(a_m\) or \(a_n\) have no coordinate derivative and therefore take the form
\eqref{eq:adapted-lower-term}.

The coefficients are products of \(d_\Phi^{-1}J\), \(q^{-1}\),
\(\beta_y\), \(\beta_s\), and pullbacks of \(a_m,a_n\).  Their required
\(W^{1,\infty}\)-bounds follow from the controlled-covering assumptions and
\(m\in W^{2,\infty}\).
\end{proof}

\subsection{Logarithmic compatibility estimate}
\label{subsec:logarithmic-compatibility}

For \(r\in H^1(\T^2)\), define
\[
  \mathcal L_m(r)
  :=
  \begin{cases}
  \displaystyle
  1+\log\left(
    1+
    \frac{\|r\|_{H^1(\T^2)}^2}
         {|r|_{L^2(\T^2)}^2}
  \right),
  & r\not\equiv0,\\[2ex]
  1,
  & r\equiv0.
  \end{cases}
\]

\begin{lemma}[Logarithmic pullback comparison]
\label{lem:log-factor-pullback}
Let \(\Phi:\widehat{\T}^2\to\T^2\) be a controlled adapted covering and set
\(R:=r\circ\Phi\).  Then
\[
  \mathcal L(R)
  \le
  C_\Phi^{\log}\mathcal L_m(r),
\]
where \(C_\Phi^{\log}\) depends only on the covering degree and
\(\mathfrak C_\Phi\).
\end{lemma}

\begin{proof}
The assertion is immediate if \(r\equiv0\).  
Otherwise,
after taking \(C_\Phi\ge1\),
Lemma~\ref{lem:cover-norm-comparison} gives
\[
  \frac{\|R\|_{H^1}^2}{|R|_{L^2}^2}
  \le
  C_\Phi^4
  \frac{\|r\|_{H^1}^2}{|r|_{L^2}^2}.
\]
Using
\[
  1+\log(1+CX)
  \le
  (1+\log C)\bigl(1+\log(1+X)\bigr),
  \qquad
  C\ge1,\quad X\ge0,
\]
we may take \(C_\Phi^{\log}=1+4\log C_\Phi\).
\end{proof}

\begin{theorem}[Logarithmic nonlinear compatibility]
\label{thm:logarithmic-nonlinear-compatibility}
Let \(m\) be a unit field in \(W^{2,\infty}(\T^2;\R^2)\) admitting a
controlled adapted covering.  Then there exists \(K_m>0\), depending only on the chosen covering and
\(\|m\|_{W^{2,\infty}}\), such that every \(u,e\in V\) satisfies
\begin{align}
  |b(e,u,e)|
  &\le
  K_m\lambda_1^{-1/2}
  \|u\|_{L^2}
  \|\mathcal A_m\|_{L^\infty}
  \|\nabla e\|_{L^2}^2
  \notag\\
  &\quad
  +
  K_m
  \|\nabla u\|_{L^2}
  \|\nabla e\|_{L^2}
  |r|_{L^2}\,
  \mathcal L_m(r)^{1/2}.
  \label{eq:log-compat-before-young}
\end{align}
Consequently, for every \(\varepsilon>0\), one has
\begin{align}
  |b(e,u,e)|
  &\le
  \left(
    K_m\lambda_1^{-1/2}
    \|u\|_{L^2}
    \|\mathcal A_m\|_{L^\infty}
    +
    \varepsilon
  \right)
  \|\nabla e\|_{L^2}^2
  \notag\\
  &\quad
  +
  \frac{K_m}{\varepsilon}
  \|\nabla u\|_{L^2}^2
  \mathcal L_m(r)
  |r|_{L^2}^2.
  \label{eq:log-compat-after-young}
\end{align}
\end{theorem}

\begin{proof}
We first argue for smooth solenoidal fields and then pass to \(V\) by density.
Lemma~\ref{lem:moving-frame-trilinear-expansion} decomposes the trilinear term
as
\[
  b(e,u,e)=I_{\rm def}+I_{\rm obs},
  \qquad
  I_{\rm def}
  :=
  -\int_{\T^2}u\cdot\mathcal A_m\,w^2\,dx,
\]
where \(I_{\rm obs}\) is the sum of the three terms containing the observed
factor \(r\).

For the defect term, H\"older's and Ladyzhenskaya's inequalities give
\[
  |I_{\rm def}|
  \le
  C
  \|\mathcal A_m\|_{L^\infty}
  \|u\|_{L^2}
  |w|_{L^2}
  \|w\|_{H^1}.
\]
Since \(w=e\cdot n\), we have
\[
  |w|_{L^2}\le |e|_{L^2},
  \qquad
  \|w\|_{H^1}
  \le
  C_m\|e\|_{H^1}.
\]
Poincare's inequality therefore yields
\[
  |I_{\rm def}|
  \le
  K_m\lambda_1^{-1/2}
  \|u\|_{L^2}
  \|\mathcal A_m\|_{L^\infty}
  \|\nabla e\|_{L^2}^2.
\]

For the observed terms, Lemma~\ref{lem:adapted-admissible-terms} and
Lemma~\ref{lem:weighted-log-trilinear} give
\[
  |I_{\rm obs}|
  \le
  C_{\Phi,m}
  \|e\|_{H^1(\T^2)}
  \|u\|_{H^1(\T^2)}
  |R|_{L^2(\widehat{\T}^2)}
  \mathcal L(R)^{1/2}.
\]
The norm comparison and logarithmic pullback lemmas imply
\[
  |R|_{L^2(\widehat{\T}^2)}
  \le
  C_\Phi |r|_{L^2(\T^2)},
  \qquad
  \mathcal L(R)^{1/2}
  \le
  (C_\Phi^{\log})^{1/2}
  \mathcal L_m(r)^{1/2}.
\]
Using Poincare's inequality for \(e,u\in H\), we obtain
\[
  |I_{\rm obs}|
  \le
  K_m
  \|\nabla u\|_{L^2}
  \|\nabla e\|_{L^2}
  |r|_{L^2}\,
  \mathcal L_m(r)^{1/2}.
\]
Combining the two bounds proves
\eqref{eq:log-compat-before-young}.  Young's inequality gives
\eqref{eq:log-compat-after-young}, after enlarging \(K_m\) if necessary.
\end{proof}

\begin{remark}[Dependence of the compatibility constant]
The constant \(K_m\) depends on the chosen adapted covering.  Inspection of the
proof shows that it may be chosen uniformly over any family for which the
covering degrees, the controlled-covering constants \(\mathfrak C_\Phi\), and
the norms \(\|m\|_{W^{2,\infty}}\) are uniformly bounded.  In particular, a
uniform choice is available for the family \(m_\varepsilon\) of
Example~\ref{ex:explicit-sinusoidal-perturbation} when
\(|\varepsilon|\le\frac12\).
\end{remark}

\begin{corollary}[Shear-defect absorbed logarithmic compatibility]
\label{cor:defect-absorbed-log-compat}
Assume the hypotheses of
Theorem~\ref{thm:logarithmic-nonlinear-compatibility}.  If
\begin{equation}
\label{eq:defect-absorption-condition}
  K_m\lambda_1^{-1/2}
  \|u\|_{L^2}
  \|\mathcal A_m\|_{L^\infty}
  \le
  \frac{\nu}{4},
\end{equation}
then, with \(r=e\cdot m\), one has
\begin{equation}
\label{eq:defect-absorbed-log-compat}
  |b(e,u,e)|
  \le
  \frac{\nu}{2}\|\nabla e\|_{L^2}^2
  +
  \frac{K_m}{\nu}
  \|\nabla u\|_{L^2}^2
  \mathcal L_m(r)|r|_{L^2}^2.
\end{equation}
\end{corollary}

\begin{proof}
Apply \eqref{eq:log-compat-after-young} with
\(\varepsilon=\nu/4\) and use
\eqref{eq:defect-absorption-condition}.
\end{proof}

\begin{remark}[Pure transverse errors]
If \(r=e\cdot m\equiv0\), then \(e=wn\), and the moving-frame expansion reduces
to
\[
  b(e,u,e)
  =
  -\int_{\T^2}
    u\cdot\mathcal A_m\,|e|^2\,dx.
\]
Thus \(\mathcal A_m\) is precisely the obstruction to the pure-shear
cancellation that holds for constant projection directions.
\end{remark}

\section{Exact synchronization in the regular resonant regime}
\label{sec:conv}

We now combine the exact error identity with the logarithmic compatibility
estimate of Section~\ref{sec:energy-estimates}.  Throughout this section, the
time-independent unit field \(m\) admits a controlled adapted covering, and we
retain the notation \(e=v-u\) and \(r=e\cdot m\) introduced in
\S\ref{ss:error-identities}.  The synchronization theorem will additionally require the
shear defect to be viscously absorbable.

\subsection{A scalar logarithmic absorption lemma}

The following elementary inequality converts the logarithmic compatibility
bound into a large-gain estimate.

\begin{lemma}[Logarithmic absorption]
\label{lem:log-absorption}
Let \(A\ge0\), \(C_0>0\), \(\delta>0\), \(X\ge0\), and \(Z>0\).  Then
\begin{equation}
\label{eq:log-absorption}
  AZ
  \left(
    1+\log\left(1+C_0\frac{X}{Z}\right)
  \right)
  \le
  \delta X
  +
  AZ
  \left(
    1+\log\left(1+\frac{C_0A}{\delta}\right)
  \right).
\end{equation}
\end{lemma}

\begin{proof}
The assertion is immediate when \(A=0\).  Otherwise, set
\(t:=C_0X/Z\) and \(\alpha:=C_0A/\delta\).  It remains to use
\[
  \log(1+t)-\frac{t}{\alpha}
  \le
  \log(1+\alpha),
  \qquad
  t\ge0,
\]
which follows by maximizing the left-hand side over \(t\ge0\).
\end{proof}

We shall also use the frame estimate
\begin{equation}
\label{eq:frame-scalar-H1-control}
  \|r\|_{H^1(\T^2)}^2
  \le
  C_{\rm fr}\|\nabla e\|_{L^2}^2,
\end{equation}
where \(C_{\rm fr}\) depends only on
\(\|m\|_{W^{1,\infty}}\) and the Poincare constant.  Indeed,
\(|r|_{L^2}\le|e|_{L^2}\) and
\[
  \|\nabla r\|_{L^2}
  \le
  \|\nabla e\|_{L^2}
  +
  \|\nabla m\|_{L^\infty}|e|_{L^2},
\]
so \eqref{eq:frame-scalar-H1-control} follows from Poincare's inequality.

\subsection{The differential inequality}

At any time \(t\) for which the shear-defect absorbability condition
\begin{equation}
\label{eq:resonant-defect-smallness-pointwise}
  K_m\lambda_1^{-1/2}
  |u(t)|
  \|\mathcal A_m\|_{L^\infty}
  \le
  \frac{\nu}{4}
\end{equation}
holds, Corollary~\ref{cor:defect-absorbed-log-compat} and the exact error
identity \eqref{eq:L2-error-identity} give
\begin{equation}
\label{eq:resonant-before-log-absorb}
  \frac{d}{dt}|e|^2
  +
  \nu\|\nabla e\|_{L^2}^2
  +
  2\mu |r|^2
  \le
  A(t)\mathcal L_m(r)|r|^2,
\end{equation}
where
\begin{equation}
\label{eq:resonant-A-def}
  A(t)
  :=
  \frac{2K_m}{\nu}
  \|\nabla u(t)\|_{L^2}^2.
\end{equation}

For \(r\not\equiv0\), the frame estimate
\eqref{eq:frame-scalar-H1-control} implies
\[
  \mathcal L_m(r)
  \le
  1+\log\left(
    1+
    C_{\rm fr}
    \frac{\|\nabla e\|_{L^2}^2}{|r|^2}
  \right).
\]
Lemma~\ref{lem:log-absorption}, applied with
\(X=\|\nabla e\|_{L^2}^2\), \(Z=|r|^2\), and
\(\delta=\nu/2\), therefore gives
\[
  A(t)\mathcal L_m(r)|r|^2
  \le
  \frac{\nu}{2}\|\nabla e\|_{L^2}^2
  +
  B(t)|r|^2,
\]
where
\begin{equation}
\label{eq:resonant-B-def}
  B(t)
  :=
  A(t)
  \left[
    1+\log\left(
      1+\frac{2C_{\rm fr}A(t)}{\nu}
    \right)
  \right].
\end{equation}
The same bound is immediate when \(r\equiv0\).  Substitution into
\eqref{eq:resonant-before-log-absorb} yields
\begin{equation}
\label{eq:resonant-basic-diff-ineq}
  \frac{d}{dt}|e|^2
  +
  \frac{\nu}{2}\|\nabla e\|_{L^2}^2
  +
  \bigl(2\mu-B(t)\bigr)|r|^2
  \le
  0.
\end{equation}

\subsection{Main convergence theorem}

We now formulate synchronization in terms of uniform reference-flow bounds and
the geometry of the projection field.

\begin{theorem}[Synchronization under controlled adapted geometry]
\label{thm:periodic-conv}
Let \(u\) be a strong solution of the periodic two-dimensional
Navier--Stokes equations.  Fix \(t_0\ge0\) and assume the uniform bounds
\[
  U_0
  :=
  \esssup_{t\ge t_0}|u(t)|
  <\infty,
  \qquad
  U_1
  :=
  \esssup_{t\ge t_0}\|\nabla u(t)\|_{L^2}
  <\infty.
\]
Let \(m\in W^{2,\infty}(\T^2;\R^2)\) be a time-independent unit field
admitting a controlled adapted covering.  Assume that its shear defect satisfies
\begin{equation}
\label{eq:resonant-defect-smallness}
  K_m\lambda_1^{-1/2}
  U_0
  \|\mathcal A_m\|_{L^\infty}
  \le
  \frac{\nu}{4},
\end{equation}
where \(K_m\) is the constant in
Corollary~\ref{cor:defect-absorbed-log-compat}.  Define
\begin{equation}
\label{eq:resonant-Bstar-def}
  A_*
  :=
  \frac{2K_m}{\nu}U_1^2,
  \qquad
  B_*
  :=
  A_*
  \left[
    1+\log\left(
      1+\frac{2C_{\rm fr}A_*}{\nu}
    \right)
  \right].
\end{equation}
If
\begin{equation}
\label{eq:resonant-gain-condition}
  \mu\ge\frac{B_*}{2},
\end{equation}
then every strong solution \(v\) of \eqref{eq:nudged} satisfies
\begin{equation}
\label{eq:resonant-decay}
  |v(t)-u(t)|^2
  \le
  \exp\left(
    -\frac{\nu\lambda_1}{2}(t-t_0)
  \right)
  |v(t_0)-u(t_0)|^2,
  \qquad
  t\ge t_0.
\end{equation}
In particular, \(v\) synchronizes exponentially with \(u\) in \(H\).
\end{theorem}

\begin{proof}
The bound \eqref{eq:resonant-defect-smallness} implies
\eqref{eq:resonant-defect-smallness-pointwise} for almost every
\(t\ge t_0\).  Moreover, \(A(t)\le A_*\), and the function
\[
  A
  \longmapsto
  A\left[
    1+\log\left(
      1+\frac{2C_{\rm fr}A}{\nu}
    \right)
  \right]
\]
is increasing on \([0,\infty)\).  Hence \(B(t)\le B_*\).  Under
\eqref{eq:resonant-gain-condition}, the differential inequality
\eqref{eq:resonant-basic-diff-ineq} therefore gives
\[
  \frac{d}{dt}|e|^2
  +
  \frac{\nu}{2}\|\nabla e\|_{L^2}^2
  \le
  0
  \qquad\text{for a.e. }t\ge t_0.
\]
Poincare's inequality and Gronwall's lemma yield
\eqref{eq:resonant-decay}.
\end{proof}

\begin{corollary}[Constant rational directions]
\label{cor:constant-rational-FLT}
Let the reference solution satisfy the hypotheses of
Theorem~\ref{thm:periodic-conv}, and let \(m=a\) be a constant rational
direction.  By Example~\ref{ex:constant-rational-directions}, \(m\) admits a
controlled adapted covering, while \(\mathcal A_m=0\).  Hence the
shear-defect condition is automatic, and every gain satisfying
\eqref{eq:resonant-gain-condition} yields exponential synchronization.
For \(a=e_1\) or \(a=e_2\), this recovers the exact-observation periodic
one-component FLT mechanism.
\end{corollary}

\begin{remark}[Geometry and viscous absorbability]
The controlled adapted covering is the structural hypothesis that enables the
logarithmic estimate.  By contrast,
\eqref{eq:resonant-defect-smallness} controls the only nonlinear contribution
without a factor of \(e\cdot m\), and is therefore a quantitative viscous
absorbability condition.

For a trajectory on the periodic global attractor, the standard energy bound
gives
\[
  U_0
  \le
  \frac{|f|}{\nu\lambda_1}.
\]
Together with
\(\|\mathcal A_m\|_{L^\infty}
\le C\|\nabla m\|_{L^\infty}\),
this shows that the allowable variation of the projection direction decreases
as the forcing, and hence the characteristic size of the reference flow,
increases.  Equivalently, the condition becomes more restrictive at larger
Grashof number.
\end{remark}

\section{Kernel-free scalar projectional observations}
\label{sec:kernel-free}

We now consider the complementary case
\(\mathcal K_m=\{0\}\): no nonzero solenoidal field is invisible to the scalar
projection.  Here synchronization requires neither adapted coordinates nor
control of the shear defect.  Instead, qualitative injectivity and the compact
embedding \(V\hookrightarrow H\) yield an observability estimate with an
arbitrarily small viscous remainder.

\subsection{Compact observability}

Kernel-freeness has the following quantitative consequence by compactness.

\begin{lemma}[Compact observability with viscous leakage]
\label{lem:kernel-free-compact-observability}
Let \(m\in L^\infty(\T^2;\R^2)\) satisfy \(|m|=1\) almost everywhere and
\[
  \mathcal K_m
  =
  \{e\in H:e\cdot m=0\}
  =
  \{0\}.
\]
Then, for every \(\eta>0\), there exists \(C_\eta<\infty\) such that
\begin{equation}
\label{eq:kernel-free-compact-observability}
  |e|^2
  \le
  C_\eta|e\cdot m|^2
  +
  \eta\|e\|^2,
  \qquad
  e\in V.
\end{equation}
\end{lemma}

\begin{proof}
Suppose that \eqref{eq:kernel-free-compact-observability} fails for some
\(\eta>0\).  For each \(j\ge1\), choose \(e_j\in V\) and normalize it so that
\[
  |e_j|=1,
  \qquad
  j|e_j\cdot m|^2+\eta\|e_j\|^2<1.
\]
Then \((e_j)\) is bounded in \(V\), while
\(|e_j\cdot m|\to0\).  By compactness of \(V\hookrightarrow H\), after
passing to a subsequence, we have
\[
  e_j\to e
  \quad\text{strongly in }H
\]
for some \(e\in H\), and the normalization gives \(|e|=1\).

Since multiplication by \(m\) is bounded on \(L^2\),
\[
  |e\cdot m|
  \le
  |(e-e_j)\cdot m|+|e_j\cdot m|
  \longrightarrow0.
\]
Thus \(e\in\mathcal K_m=\{0\}\), contradicting \(|e|=1\).
\end{proof}

The proof is qualitative, so \(C_\eta\) is not explicit and may deteriorate rapidly as \(\eta\downarrow0\). For constant irrational directions, this loss encodes the underlying small-divisor geometry.

\subsection{Exact synchronization in the kernel-free regime}

The compact-observability estimate is applied to the nonlinear production term,
so its viscous leakage is weighted by a reference-flow bound rather than by the
nudging gain.  Let \(C_{\rm nl}\) be such that
\begin{equation}
\label{eq:kernel-free-standard-nonlinear}
  |b(e,u,e)|
  \le
  \frac{\nu}{4}\|e\|^2
  +
  \frac{C_{\rm nl}}{\nu}\|u\|^2|e|^2.
\end{equation}

\begin{theorem}[Exact synchronization in the kernel-free regime]
\label{thm:kernel-free-exact-conv}
Let \(u\) be a strong solution of the periodic two-dimensional
Navier--Stokes equations.  Fix \(t_0\ge0\) and assume
\[
  U_1
  :=
  \esssup_{t\ge t_0}\|u(t)\|
  <\infty.
\]
Let \(m\in L^\infty(\T^2;\R^2)\) be time-independent, satisfy
\(|m|=1\) almost everywhere, and be kernel-free on \(H\).  Set
\begin{equation}
\label{eq:kernel-free-Lambda-star}
  \Lambda_*
  :=
  \frac{C_{\rm nl}}{\nu}U_1^2.
\end{equation}
Choose \(\eta>0\) so that
\begin{equation}
\label{eq:kernel-free-eta-condition}
  \Lambda_*\eta
  \le
  \frac{\nu}{4},
\end{equation}
and let \(C_\eta\) be the corresponding constant in
Lemma~\ref{lem:kernel-free-compact-observability}.  If
\begin{equation}
\label{eq:kernel-free-gain-condition}
  \mu
  \ge
  \Lambda_*C_\eta,
\end{equation}
then every strong solution \(v\) of \eqref{eq:nudged} satisfies
\begin{equation}
\label{eq:kernel-free-decay}
  |v(t)-u(t)|^2
  \le
  e^{-\nu\lambda_1(t-t_0)}
  |v(t_0)-u(t_0)|^2,
  \qquad
  t\ge t_0.
\end{equation}
Thus \(v\) synchronizes exponentially with \(u\) in \(H\).
\end{theorem}

\begin{proof}
The exact error identity and
\eqref{eq:kernel-free-standard-nonlinear} give, for almost every
\(t\ge t_0\),
\[
  \frac12\frac{d}{dt}|e|^2
  +
  \frac{3\nu}{4}\|e\|^2
  +
  \mu|r|^2
  \le
  \Lambda_*|e|^2.
\]
Compact observability yields
\[
  \Lambda_*|e|^2
  \le
  \Lambda_*C_\eta|r|^2
  +
  \Lambda_*\eta\|e\|^2.
\]
Using \eqref{eq:kernel-free-eta-condition} and
\eqref{eq:kernel-free-gain-condition}, we obtain
\[
  \frac12\frac{d}{dt}|e|^2
  +
  \frac{\nu}{2}\|e\|^2
  \le
  0.
\]
Poincare's inequality and Gronwall's lemma give
\eqref{eq:kernel-free-decay}.
\end{proof}

The parameter \(\eta\) is chosen independently of \(\mu\).  Consequently, the
leakage term is multiplied by \(\Lambda_*\), not by the gain, and
\eqref{eq:kernel-free-gain-condition} is purely a lower bound: every larger
gain is also admissible.

\subsection{Constant projection directions}

For a constant field, kernel-freeness is exactly the lattice-nonresonance
condition complementary to Example~\ref{ex:constant-rational-directions}.

\begin{proposition}[Fourier criterion for constant directions]
\label{prop:constant-direction-kernel-free}
Let \(m=a\in\mathbb S^1\) be constant and set \(n:=a^\perp\).  Then the
projection \(e\mapsto e\cdot a\) is kernel-free on \(H\) if and only if
\[
  k\cdot n\neq0
  \qquad
  \text{for every }k\in\mathbb Z^2\setminus\{0\}.
\]
Equivalently, \(n\) is not parallel to any nonzero lattice vector.  Thus
constant directions with irrational transverse slope are kernel-free, whereas
rational transverse directions admit invisible shear currents.
\end{proposition}

\begin{proof}
If \(e\in\mathcal K_a\), then \(e=\psi n\) for some
\(\psi\in L^2(\T^2)\).  The divergence-free and mean-zero conditions imply
\[
  n\cdot\nabla\psi=0,
  \qquad
  \widehat\psi_0=0.
\]
Writing
\(\psi=\sum_{k\in\mathbb Z^2}\widehat\psi_k e^{ik\cdot x}\), we obtain
\[
  (k\cdot n)\widehat\psi_k=0
  \qquad
  \text{for every }k\in\mathbb Z^2.
\]
The stated nonresonance condition therefore forces \(\psi=0\), and hence
\(e=0\).

Conversely, if \(k\cdot n=0\) for some nonzero lattice vector \(k\), then
\[
  e(x):=\cos(k\cdot x)n
\]
is nonzero, periodic, divergence-free, mean-zero, and satisfies \(e\cdot a=0\).
Thus the projection is not kernel-free.  Finally, \(k\cdot n=0\) precisely
when \(n\) is parallel to the lattice vector \(k^\perp\).
\end{proof}

Consequently, if \(a^\perp\) is not parallel to a lattice vector, then
Theorem~\ref{thm:kernel-free-exact-conv} gives exponential synchronization for
every sufficiently large gain.

Although such a direction is kernel-free, it is not uniformly coercive on
\(H\).  On a divergence-free Fourier mode of frequency \(k\), the ratio of the
observed amplitude to the full amplitude is
\({|k\cdot n|}/{|k|}\),
which can approach zero along high frequencies.  The leakage term in
\eqref{eq:kernel-free-compact-observability} leaves these poorly observed modes
to viscosity; correspondingly, \(C_\eta\) reflects the small-divisor geometry
of the direction.

\section{Type-I coarse scalar projectional observations}
\label{sec:coarse}

We now replace the exact scalar discrepancy \(r=e\cdot m\) by its coarse
observation \(I_hr\).  In both the regular resonant and kernel-free regimes,
the Type-I approximation estimate converts the feedback form
\((I_hr,r)\) into exact scalar damping, up to a derivative leakage absorbed by
viscosity under the usual gain--resolution condition.  Only the treatment of
the nonlinear term differs between the two regimes.

\subsection{Scalar Type-I observations}
\label{subsec:type-I-observations}

Let \(I_h:L^2(\T^2)\to L^2(\T^2)\) be a family of linear scalar observation
operators indexed by \(h>0\).  
We assume that
\begin{align}
\label{eq:coarse-Ih-stability}
  \|I_h\phi\|_{L^2}
  &\le
  c_0\|\phi\|_{L^2},
  && \phi\in L^2(\T^2),
  \\
\label{eq:coarse-Ih-approx}
  \|\phi-I_h\phi\|_{L^2}
  &\le
  c_1h\|\nabla\phi\|_{L^2},
  && \phi\in H^1(\T^2),
\end{align}
with constants \(c_0,c_1>0\) independent of \(h\).
These are the scalar Type-I conditions.  They are satisfied, for example, by
low Fourier projections, local or mollified volume averages, and suitable
\(L^2\)-stable finite-element or mesh-free quasi-interpolants.  The stability
bound also ensures that the coarse feedback operator is bounded on \(L^2\), so
Proposition~\ref{prop:bounded-feedback-wellposed} gives strong well-posedness of
\eqref{eq:coarse-nudged} for every fixed \(h\).

Raw nodal interpolation is not included: point evaluation is neither continuous
on \(L^2(\T^2)\) nor defined for a general \(H^1(\T^2)\) function in two
dimensions.  Nodal data must instead be incorporated through an \(L^2\)-stable
reconstruction or a higher-order Type-II framework, which we do not pursue here.

\subsection{Common scalar coercivity}
\label{subsec:common-scalar-coercivity}

For coarse observations, the exact dissipation \(|r|^2\) is replaced by
\((I_hr,r)\).  The Type-I approximation property recovers scalar coercivity up
to a viscous derivative leakage.
For \(m\in W^{1,\infty}\), let \(C_{\rm fr}\) denote the constant in
\eqref{eq:frame-scalar-H1-control}.

\begin{lemma}[Type-I scalar feedback coercivity]
\label{lem:coarse-feedback-coercivity}
Let \(m\in W^{1,\infty}(\T^2;\R^2)\) be a unit field, let \(e\in V\), and set
\(r:=e\cdot m\).  Then
\begin{equation}
\label{eq:coarse-feedback-coercivity}
  (I_hr,r)
  \ge
  \frac12|r|_{L^2}^2
  -
  \frac{c_1^2C_{\rm fr}}{2}
  h^2\|\nabla e\|_{L^2}^2.
\end{equation}
\end{lemma}

\begin{proof}
Using \eqref{eq:coarse-Ih-approx}, we infer
\[
  (I_hr,r)
  =
  |r|_{L^2}^2-\bigl((I-I_h)r,r\bigr)
  \ge
  |r|_{L^2}^2
  -
  c_1h\|\nabla r\|_{L^2}|r|_{L^2}.
\]
The frame estimate
\(\|\nabla r\|_{L^2}^2
\le C_{\rm fr}\|\nabla e\|_{L^2}^2\)
and Young's inequality give
\eqref{eq:coarse-feedback-coercivity}.
\end{proof}

Only the approximation property \eqref{eq:coarse-Ih-approx} is used in this
coercivity estimate.  The stability condition
\eqref{eq:coarse-Ih-stability} instead guarantees that the feedback is bounded
on \(L^2\).  No positivity or self-adjointness of \(I_h\) is required.

\subsection{Type-I synchronization in the regular resonant regime}
\label{subsec:type-I-resonant}

We combine the common Type-I coercivity estimate with the logarithmic
compatibility mechanism of Sections~\ref{sec:energy-estimates}
and~\ref{sec:conv}.  The geometric hypotheses are the same as in the
exact-observation theorem.

\begin{theorem}[Type-I synchronization in the regular resonant regime]
\label{thm:coarse-projectional-conv}
Let \(u\) be a strong solution of the periodic two-dimensional
Navier--Stokes equations.  Fix \(t_0\ge0\) and assume
\[
  U_0
  :=
  \esssup_{t\ge t_0}|u(t)|
  <\infty,
  \qquad
  U_1
  :=
  \esssup_{t\ge t_0}\|\nabla u(t)\|_{L^2}
  <\infty.
\]
Let \(m\in W^{2,\infty}(\T^2;\R^2)\) be a time-independent unit field
admitting a controlled adapted covering, and assume
\begin{equation}
\label{eq:coarse-defect-smallness}
  K_m\lambda_1^{-1/2}
  U_0
  \|\mathcal A_m\|_{L^\infty}
  \le
  \frac{\nu}{4}.
\end{equation}
Define
\begin{equation}
\label{eq:coarse-Bstar-def}
  A_*
  :=
  \frac{2K_m}{\nu}U_1^2,
  \qquad
  B_*^{\rm c}
  :=
  A_*
  \left[
    1+\log\left(
      1+\frac{4C_{\rm fr}A_*}{\nu}
    \right)
  \right].
\end{equation}
Let \(I_h\) satisfy
\eqref{eq:coarse-Ih-stability}--\eqref{eq:coarse-Ih-approx}.  If
\begin{equation}
\label{eq:coarse-gain-lower}
  \mu\ge B_*^{\rm c},
\end{equation}
and
\begin{equation}
\label{eq:coarse-resolution-condition}
  c_1^2C_{\rm fr}\,\mu h^2
  \le
  \frac{\nu}{4},
\end{equation}
then the strong solution \(v\) of \eqref{eq:coarse-nudged} satisfies
\begin{equation}
\label{eq:coarse-decay}
  |v(t)-u(t)|^2
  \le
  \exp\left(
    -\frac{\nu\lambda_1}{2}(t-t_0)
  \right)
  |v(t_0)-u(t_0)|^2,
  \qquad
  t\ge t_0.
\end{equation}
\end{theorem}

\begin{proof}
The coarse error identity and
Lemma~\ref{lem:coarse-feedback-coercivity} give
\[
  \frac{d}{dt}|e|^2
  +
  \left(
    2\nu-c_1^2C_{\rm fr}\,\mu h^2
  \right)
  \|\nabla e\|_{L^2}^2
  +
  \mu|r|_{L^2}^2
  \le
  2|b(e,u,e)|.
\]
By \eqref{eq:coarse-defect-smallness} and
Corollary~\ref{cor:defect-absorbed-log-compat}, we have
\[
  2|b(e,u,e)|
  \le
  \nu\|\nabla e\|_{L^2}^2
  +
  A(t)\mathcal L_m(r)|r|_{L^2}^2,
  \qquad
  A(t):=
  \frac{2K_m}{\nu}\|\nabla u(t)\|_{L^2}^2.
\]
Using \eqref{eq:coarse-resolution-condition}, we therefore obtain
\[
  \frac{d}{dt}|e|^2
  +
  \frac{3\nu}{4}\|\nabla e\|_{L^2}^2
  +
  \mu|r|_{L^2}^2
  \le
  A(t)\mathcal L_m(r)|r|_{L^2}^2.
\]

The frame estimate \eqref{eq:frame-scalar-H1-control} and
Lemma~\ref{lem:log-absorption}, with \(\delta=\nu/4\), imply
\[
  A(t)\mathcal L_m(r)|r|_{L^2}^2
  \le
  \frac{\nu}{4}\|\nabla e\|_{L^2}^2
  +
  B_{\rm c}(t)|r|_{L^2}^2,
\]
where
\[
  B_{\rm c}(t)
  :=
  A(t)
  \left[
    1+\log\left(
      1+\frac{4C_{\rm fr}A(t)}{\nu}
    \right)
  \right].
\]
This remains valid when \(r\equiv0\).  Hence
\[
  \frac{d}{dt}|e|^2
  +
  \frac{\nu}{2}\|\nabla e\|_{L^2}^2
  +
  \bigl(\mu-B_{\rm c}(t)\bigr)|r|_{L^2}^2
  \le
  0.
\]
Since \(A(t)\le A_*\) and the defining function of \(B_{\rm c}\) is
increasing, \(B_{\rm c}(t)\le B_*^{\rm c}\).  Condition
\eqref{eq:coarse-gain-lower}, followed by Poincare's inequality and
Gronwall's lemma, proves \eqref{eq:coarse-decay}.
\end{proof}

\begin{remark}[Comparison with exact observations]
Let \(B_*\) be the exact-observation quantity defined in
\eqref{eq:resonant-Bstar-def}.  Since
\[
  1+4x\le(1+2x)^2,
  \qquad
  x\ge0,
\]
one has \(B_*^{\rm c}\le2B_*\).  Thus the sufficient coarse lower-gain
threshold is at most four times the sufficient exact threshold
\(\mu\ge B_*/2\).  The essential new restriction is the Type-I
gain--resolution balance
\[
  \mu h^2
  \lesssim
  \nu,
\]
which prevents the interpolation leakage from overwhelming viscosity.
\end{remark}

\subsection{Type-I synchronization in the kernel-free regime}
\label{subsec:type-I-kernel-free}

We now combine Type-I scalar coercivity with compact observability.  The
argument is direct from the coarse error identity and does not pass through the
exact-observation theorem.

\begin{theorem}[Type-I synchronization in the kernel-free regime]
\label{thm:coarse-kernel-free-conv}
Let \(u\) be a strong solution of the periodic two-dimensional
Navier--Stokes equations.  Fix \(t_0\ge0\) and assume
\[
  U_1
  :=
  \esssup_{t\ge t_0}\|u(t)\|
  <\infty.
\]
Let \(m\in W^{1,\infty}(\T^2;\R^2)\) be a time-independent unit field that
is kernel-free on \(H\), and define
\begin{equation}
\label{eq:coarse-kernel-free-Lambda}
  \Lambda_*
  :=
  \frac{C_{\rm nl}}{\nu}U_1^2.
\end{equation}
Choose \(\eta>0\) so that
\begin{equation}
\label{eq:coarse-kernel-free-eta}
  \Lambda_*\eta
  \le
  \frac{\nu}{4},
\end{equation}
and let \(C_\eta\) be the corresponding constant in
\eqref{eq:kernel-free-compact-observability}.  Let \(I_h\) satisfy
\eqref{eq:coarse-Ih-stability}--\eqref{eq:coarse-Ih-approx}.  If
\begin{equation}
\label{eq:coarse-kernel-free-gain}
  \mu
  \ge
  2\Lambda_*C_\eta,
\end{equation}
and
\begin{equation}
\label{eq:coarse-kernel-free-resolution}
  c_1^2C_{\rm fr}\,\mu h^2
  \le
  \frac{\nu}{4},
\end{equation}
then the strong solution \(v\) of \eqref{eq:coarse-nudged} satisfies
\begin{equation}
\label{eq:coarse-kernel-free-decay}
  |v(t)-u(t)|^2
  \le
  \exp\left(
    -\frac{\nu\lambda_1}{2}(t-t_0)
  \right)
  |v(t_0)-u(t_0)|^2,
  \qquad
  t\ge t_0.
\end{equation}
\end{theorem}

\begin{proof}
The coarse error identity,
Lemma~\ref{lem:coarse-feedback-coercivity}, and
\eqref{eq:kernel-free-standard-nonlinear} give
\[
  \frac12\frac{d}{dt}|e|^2
  +
  \left(
    \frac{3\nu}{4}
    -
    \frac{c_1^2C_{\rm fr}}{2}\mu h^2
  \right)
  \|\nabla e\|_{L^2}^2
  +
  \frac{\mu}{2}|r|_{L^2}^2
  \le
  \Lambda_*|e|^2.
\]
Compact observability implies
\[
  \Lambda_*|e|^2
  \le
  \Lambda_*C_\eta|r|_{L^2}^2
  +
  \Lambda_*\eta\|\nabla e\|_{L^2}^2.
\]
Consequently, we get
\[
  \frac12\frac{d}{dt}|e|^2
  +
  \left(
    \frac{3\nu}{4}
    -
    \frac{c_1^2C_{\rm fr}}{2}\mu h^2
    -
    \Lambda_*\eta
  \right)
  \|\nabla e\|_{L^2}^2
  +
  \left(
    \frac{\mu}{2}
    -
    \Lambda_*C_\eta
  \right)
  |r|_{L^2}^2
  \le
  0.
\]
Conditions \eqref{eq:coarse-kernel-free-eta}--%
\eqref{eq:coarse-kernel-free-resolution} make the two coefficients at least
\(3\nu/8\) and \(0\), respectively.  Hence
\[
  \frac12\frac{d}{dt}|e|^2
  +
  \frac{3\nu}{8}\|\nabla e\|_{L^2}^2
  \le
  0.
\]
Poincare's inequality and Gronwall's lemma imply
\eqref{eq:coarse-kernel-free-decay}.
\end{proof}

\begin{remark}[Gain--resolution window and common Type-I layer]
The lower bound \eqref{eq:coarse-kernel-free-gain} comes from compact
observability, whereas \eqref{eq:coarse-kernel-free-resolution} is the upper
gain--resolution restriction created by Type-I interpolation leakage.  The two
conditions are compatible for sufficiently small \(h\); for instance, it is
enough that
\[
  2c_1^2C_{\rm fr}\Lambda_*C_\eta h^2
  \le
  \frac{\nu}{4}.
\]

The resonant and kernel-free Type-I theorems share only the scalar coercivity
Lemma~\ref{lem:coarse-feedback-coercivity}.  Their exact-data mechanisms remain
different: logarithmic control of an invisible shear family in the regular
resonant case, and compact observability of the full error in the kernel-free
case.  Coarse observation introduces no new global geometric hypothesis, but
the Type-I approximation step requires the additional regularity
\(m\in W^{1,\infty}\).
\end{remark}

\section{Parabolic upgrade from
\texorpdfstring{\(L^2\)}{L2} to
\texorpdfstring{\(H^1\)}{H1} synchronization}
\label{sec:V-decay}

The preceding four convergence theorems establish exponential synchronization
in \(H\).  On the periodic two-dimensional domain, this decay automatically
upgrades to \(V\)-decay.  The argument uses the enstrophy cancellation and only
the \(L^2\)-boundedness of the feedback; no coercivity of the feedback at the
\(V\)-level is required.

\begin{theorem}[Upgrade from \(H\) to \(V\)]
\label{thm:H-to-V-upgrade}
Assume the hypotheses of one of
Theorems~\ref{thm:periodic-conv},
\ref{thm:kernel-free-exact-conv},
\ref{thm:coarse-projectional-conv}, or
\ref{thm:coarse-kernel-free-conv}.  Write the corresponding \(H\)-decay
estimate as
\begin{equation}
\label{eq:H-decay-for-upgrade}
  |e(t)|^2
  \le
  C_H e^{-\gamma_H(t-t_0)},
  \qquad
  t\ge t_0.
\end{equation}
Thus \(\gamma_H=\nu\lambda_1\) in
Theorem~\ref{thm:kernel-free-exact-conv}, and
\(\gamma_H=\nu\lambda_1/2\) in the other three cases.

Then, for every
\[
  0<\gamma_V<\min\{\nu\lambda_1,\gamma_H\},
\]
there exists \(C_V<\infty\) such that
\begin{equation}
\label{eq:V-decay-upgrade}
  \|v(t)-u(t)\|^2
  \le
  C_V e^{-\gamma_V(t-t_0)},
  \qquad
  t\ge t_0.
\end{equation}
In particular, all exact and Type-I synchronization results of this paper hold
exponentially in both \(H\) and \(V\).
\end{theorem}

\begin{proof}
Let \(\mathcal M=M\) in the exact case and
\(\mathcal M=M_h\) in the coarse case.  In either case,
\[
  \|\mathcal M z\|_{L^2}
  \le
  C_{\rm fb}|z|,
  \qquad
  C_{\rm fb}
  =
  \begin{cases}
    1, & \mathcal M=M,\\
    c_0, & \mathcal M=M_h,
  \end{cases}
\]
where the coarse bound follows from
\eqref{eq:coarse-Ih-stability}.

For periodic two-dimensional solenoidal fields, the enstrophy cancellation
\[
  b(z,z,Az)=0
\]
holds for every \(z\in D(A)\).  Taking the coefficient of \(\theta^2\) in
the identity
\[
  b(u+\theta e,u+\theta e,A(u+\theta e))=0
\]
gives
\begin{equation}
\label{eq:polarized-enstrophy}
  b(e,e,Au)
  +
  b(e,u,Ae)
  +
  b(u,e,Ae)
  =
  0.
\end{equation}
Writing \(v=u+e\) in the error equation and testing with \(Ae\), the nonlinear
terms are
\[
  b(e,u,Ae)+b(u,e,Ae)+b(e,e,Ae).
\]
Using \(b(e,e,Ae)=0\) and \eqref{eq:polarized-enstrophy}, we therefore obtain
\begin{equation}
\label{eq:strong-error-identity}
  \frac12\frac{d}{dt}\|e\|^2
  +
  \nu|Ae|^2
  =
  b(e,e,Au)
  -
  \mu(\mathcal M e,Ae)
\end{equation}
for almost every \(t\ge t_0\).

Agmon's inequality and the interpolation estimate
\(\|e\|^2\le |e||Ae|\) imply
\[
  |b(e,e,Au)|
  \le
  \|e\|_{L^\infty}\|e\||Au|
  \le
  C|e||Ae||Au|.
\]
Moreover, we have
\[
  \mu|(\mathcal M e,Ae)|
  \le
  \mu C_{\rm fb}|e||Ae|.
\]
Applying Young's inequality in \eqref{eq:strong-error-identity}, we obtain
\begin{equation}
\label{eq:strong-error-differential}
  \frac{d}{dt}\|e\|^2
  +
  \nu|Ae|^2
  \le
  \frac{C}{\nu}
  \left(
    |Au|^2+\mu^2C_{\rm fb}^2
  \right)
  |e|^2.
\end{equation}
Since \(|Ae|^2\ge\lambda_1\|e\|^2\), the left-hand side controls the
\(V\)-energy.

It remains to control the coefficient involving \(Au\).  
Testing the reference equation \eqref{eq:NSE} with \(Au\)
and using the same enstrophy
cancellation gives
\[
  \frac{d}{dt}\|u\|^2
  +
  \nu|Au|^2
  \le
  \frac{|f|^2}{\nu}.
\]
The uniform bound \(\|u(t)\|\le U_1\) therefore implies
\begin{equation}
\label{eq:reference-Au-sliding-bound}
  \int_t^{t+1}|Au(s)|^2\,ds
  \le
  \frac{U_1^2}{\nu}
  +
  \frac{|f|^2}{\nu^2},
  \qquad
  t\ge t_0.
\end{equation}
Consequently, the coefficient on the right-hand side of
\eqref{eq:strong-error-differential} has uniformly bounded integrals over
unit time intervals.

Set \(y(t):=\|e(t)\|^2\).  Combining
\eqref{eq:strong-error-differential} with
\eqref{eq:H-decay-for-upgrade} and the spectral inequality gives
\[
  y'(t)+\nu\lambda_1y(t)
  \le
  C_H q(t)e^{-\gamma_H(t-t_0)},
\]
where
\[
  q(t)
  :=
  \frac{C}{\nu}
  \left(
    |Au(t)|^2+\mu^2C_{\rm fb}^2
  \right)
\]
has uniformly bounded integrals over intervals of length one by
\eqref{eq:reference-Au-sliding-bound}.  Variation of constants yields
\[
  y(t)
  \le
  e^{-\nu\lambda_1(t-t_0)}y(t_0)
  +
  C_H
  \int_{t_0}^t
    e^{-\nu\lambda_1(t-s)}
    q(s)e^{-\gamma_H(s-t_0)}
  \,ds.
\]
For any
\(\gamma_V<\min\{\nu\lambda_1,\gamma_H\}\), the integrand is bounded by
\[
  e^{-\gamma_V(t-t_0)}
  e^{-(\nu\lambda_1-\gamma_V)(t-s)}q(s).
\]
Partitioning \([t_0,t]\) into unit intervals and using the uniform integral
bound for \(q\), the remaining convolution is bounded independently of \(t\).
This proves \eqref{eq:V-decay-upgrade}.
\end{proof}

\begin{remark}[No additional observation hypothesis]
The upgrade uses only the exponential \(H\)-decay and the boundedness of the
feedback on \(L^2\).  It therefore introduces no new geometric assumption,
gain condition, or gain--resolution restriction.  In particular, derivatives
of the projection field and positivity of the feedback at the \(V\)-level play
no role.  The constant \(C_V\) may depend on the fixed gain \(\mu\), and in the
coarse case on the fixed observation resolution, but the admissible parameter
ranges are exactly those of the corresponding \(H\)-synchronization theorem.
\end{remark}

\section{Concluding remarks}
\label{sec:concluding-remarks}

We have established exponential synchronization in both \(H\) and \(V\) from a
single signed scalar velocity projection through two distinct mechanisms.  
In the regular resonant regime, static observability fails because invisible shear currents are present.
Synchronization is instead driven by nonlinear
compatibility: controlled adapted coordinates and a moving-frame expansion expose
FLT-type logarithmic terms involving the observed scalar, while the only purely
transverse contribution is measured by the shear defect and absorbed by
viscosity.  In the kernel-free regime, no such transfer mechanism is needed.
Qualitative injectivity and compactness yield observability with viscous
high-frequency leakage.  Since this leakage is applied only to the nonlinear
production term, exact kernel-free observations admit every sufficiently large
gain.

For constant projection directions, the distinction is arithmetic.  Rational
transverse directions support periodic invisible shear modes and belong to the
resonant mechanism, whereas nonresonant directions are kernel-free.  The explicit
sinusoidal family shows that the resonant theory is genuinely nonconstant and
persists under small controlled geometric perturbations, with the resulting
variation measured by the shear defect.

A common Type-I approximation layer extends both mechanisms to coarse scalar
data. 
Its principal new parameter restriction is the familiar gain--resolution
balance \(\mu h^2\lesssim\nu\); in the kernel-free case, the scalar
approximation step also requires \(m\in W^{1,\infty}\).
Once exponential \(H\)-synchronization is known, periodic parabolic smoothing and
the two-dimensional enstrophy cancellation upgrade it to exponential
\(V\)-synchronization for both exact and Type-I feedbacks.  This upgrade requires
only boundedness of the feedback and introduces no further geometric or
gain--resolution hypothesis.

We have restricted the coarse theory to Type-I observations in order to keep
the dichotomy between the resonant and kernel-free regimes and its two analytic mechanisms in focus.
Type-II observations, including raw nodal reconstructions, introduce
second-order interpolation leakage and require a mixed strong-energy analysis.
Their treatment is left for future work.  Other natural extensions include
localized observations, closed surfaces, bounded and nonhomogeneous-boundary
domains, and boundary-anchored rank-one geometries.

\appendix

\section{Well-posedness of bounded-feedback nudged systems}
\label{sec:wellposed}

The exact and coarse feedbacks considered in this paper are bounded linear
perturbations of the periodic two-dimensional Navier--Stokes equations:
\[
  Mz=(z\cdot m)m,
  \qquad
  M_hz=\bigl(I_h(z\cdot m)\bigr)m.
\]
For the coarse operator, it is enough that
\(I_h:L^2(\T^2)\to L^2(\T^2)\) be bounded.  None of the geometric,
kernel-free, or approximation hypotheses used for synchronization enters the
existence theory.

\begin{theorem}[Strong well-posedness with bounded feedback]
\label{thm:wellposed}
Let \(T>0\), \(v_0\in V\),
\(f,u\in L^2(0,T;H)\),
and let
\(
  \mathcal M(t)\in
  \mathcal L\bigl(L^2(\T^2)^2,L^2(\T^2)^2\bigr)
\)
be strongly measurable in \(t\), with
\begin{equation}
\label{eq:wellposed-M-bound}
  \mathfrak M_T
  :=
  \operatorname*{ess\,sup}_{0<t<T}
  \|\mathcal M(t)\|_{\mathcal L(L^2,L^2)}
  <\infty.
\end{equation}
Then, for every \(\mu\ge0\), the system
\begin{equation}
\label{eq:wellposed-nudged}
  v_t+\nu Av+B(v,v)
  =
  f-\mu P_\sigma\mathcal M(t)(v-u),
  \qquad
  v(0)=v_0,
\end{equation}
has a unique strong solution satisfying
\[
  v\in C([0,T];V)\cap L^2(0,T;D(A)),
  \qquad
  v_t\in L^2(0,T;H).
\]
Among strong solutions with initial data in \(V\), the solution depends
continuously on the initial data in the \(H\)-metric.

Consequently, the solution is global whenever \(f\), \(u\), and
\(\mathcal M\) satisfy the corresponding bounds on every finite time interval.
\end{theorem}

\begin{proof}
Let \(P_N\) be the orthogonal projection onto the first \(N\) Stokes modes.
We consider the Galerkin system
\[
  \frac{dv_N}{dt}
  +\nu Av_N
  +P_NB(v_N,v_N)
  +\mu P_NP_\sigma\mathcal M(t)v_N
  =
  P_Nf+\mu P_NP_\sigma\mathcal M(t)u,
\]
with \(v_N(0)=P_Nv_0\).
Its right-hand side is measurable in time and locally Lipschitz in \(v_N\), so
Caratheodory's theorem gives a local absolutely continuous solution.

Testing with \(v_N\), using \(b(v_N,v_N,v_N)=0\), and applying
\eqref{eq:wellposed-M-bound}, Poincare's inequality, and Young's inequality give
\[
  \frac{d}{dt}|v_N|^2
  +
  \nu\|v_N\|^2
  \le
  C_{\nu,\mu,\mathfrak M_T}
  \bigl(
    |f|^2+|u|^2+|v_N|^2
  \bigr).
\]
Gronwall's inequality therefore yields an \(N\)-independent bound in
\[
  L^\infty(0,T;H)\cap L^2(0,T;V),
\]
and in particular extends the Galerkin solution to all of \([0,T]\).

Testing instead with \(Av_N\), we use the periodic two-dimensional enstrophy
cancellation
\[
  (B(v_N,v_N),Av_N)=0.
\]
Indeed, if \(\omega_N=\operatorname{curl}v_N\), then we have
\[
  (B(v_N,v_N),Av_N)
  =
  \int_{\T^2}
    (v_N\cdot\nabla\omega_N)\omega_N\,dx
  =
  0
\]
by periodicity and \(\nabla\cdot v_N=0\).  The feedback and forcing terms are
estimated by boundedness of \(\mathcal M(t)\) and Young's inequality, giving
\[
  \frac{d}{dt}\|v_N\|^2
  +
  \nu|Av_N|^2
  \le
  C_{\nu,\mu,\mathfrak M_T}
  \bigl(
    |f|^2+|u|^2+|v_N|^2
  \bigr).
\]
Together with the \(L^2\)-estimate, this shows that
\[
  (v_N)
  \quad\text{is bounded in}\quad
  L^\infty(0,T;V)\cap L^2(0,T;D(A)).
\]

The Galerkin equation and the periodic estimate
\[
  |B(v_N,v_N)|
  \le
  C\|v_N\|_{L^\infty}\|v_N\|
  \le
  C|Av_N|\|v_N\|
\]
then imply that \((\partial_t v_N)\) is bounded in \(L^2(0,T;H)\).
The Aubin--Lions--Simon theorem \cite{Simon1987} therefore gives, after passing to a subsequence,
\[
  v_N\to v
  \quad\text{strongly in }L^2(0,T;V),
  \qquad
  v_N\rightharpoonup v
  \quad\text{weakly in }L^2(0,T;D(A)).
\]
The strong convergence is sufficient to pass to the nonlinear term.  Moreover, we have
\[
  |\mathcal M(t)(v_N-v)|
  \le
  \mathfrak M_T|v_N-v|,
\]
so the feedback term converges strongly in \(L^2(0,T;H)\).  Hence \(v\) solves
\eqref{eq:wellposed-nudged} and satisfies
\[
  v\in L^\infty(0,T;V)\cap L^2(0,T;D(A)),
  \qquad
  v_t\in L^2(0,T;H).
\]
The Lions--Magenes theorem, applied with
\(D(A)\hookrightarrow V\hookrightarrow D(A)'\), gives
\(v\in C([0,T];V)\), with \(v(0)=v_0\).

For uniqueness, let \(z=v^{(1)}-v^{(2)}\) be the difference of two strong
solutions driven by the same data.  Testing its equation with \(z\), using
\(b(v^{(2)},z,z)=0\), Ladyzhenskaya's inequality, and boundedness of
\(\mathcal M(t)\), we obtain
\[
  \frac{d}{dt}|z|^2
  \le
  \left(
    2\mu\mathfrak M_T
    +
    \frac{C}{\nu}\|v^{(1)}(t)\|^2
  \right)
  |z|^2.
\]
Since \(v^{(1)}\in L^2(0,T;V)\), Gronwall's inequality yields
\[
  |z(t)|^2
  \le
  |z(0)|^2
  \exp\left(
    2\mu\mathfrak M_Tt
    +
    \frac{C}{\nu}
    \int_0^t\|v^{(1)}(s)\|^2\,ds
  \right).
\]
This proves uniqueness and continuous dependence in the \(H\)-metric.
\end{proof}

\begin{remark}[Exact and coarse realizations]
For exact projectional feedback,
\[
  \mathcal M(t)z=(z\cdot m(t))m(t),
  \qquad
  \|\mathcal M(t)\|_{\mathcal L(L^2,L^2)}\le1.
\]
For coarse scalar feedback,
\[
  \mathcal M(t)z
  =
  \bigl(I_h(t)(z\cdot m(t))\bigr)m(t),
  \qquad
  \|\mathcal M(t)\|_{\mathcal L(L^2,L^2)}
  \le
  \|I_h(t)\|_{\mathcal L(L^2,L^2)}.
\]
Thus Theorem~\ref{thm:wellposed} applies whenever the scalar observation
operator is \(L^2\)-bounded.  Uniformity in \(h\) is needed for the coarse-data
synchronization theory, but not for well-posedness at a fixed resolution.
Likewise, the estimates establish well-posedness for each fixed gain \(\mu\);
the synchronization arguments do not require bounds uniform in \(\mu\).
\end{remark}

\section*{Acknowledgements}

Part of this work was conducted during a visit to Jiangxi Normal University, at
the invitation of Professors Ming Mei and Zejia Wang.  I am grateful to them for
their invitation and hospitality, and to the PDE community in Nanchang for the
stimulating mathematical environment during the visit.
This work was partially supported by the Natural Sciences and Engineering
Research Council of Canada (NSERC) through a Discovery Grant.


\end{document}